\documentclass[12pt]{amsart}
\usepackage{geometry}
\usepackage{amsmath}
\usepackage{amsthm}
\usepackage{bbm}
\usepackage{amsfonts,amssymb,bm}
\usepackage{fancyhdr}
\usepackage{mathrsfs}
\usepackage{appendix}
\usepackage{graphicx}
\usepackage[all]{xy}
\usepackage{color}

\usepackage{float} 
\usepackage{amsthm}

\newtheorem{thm}{Theorem}[section]
\newtheorem{defi}[thm]{Definition}
\newtheorem{lem}[thm]{Lemma}
\newtheorem{cor}[thm]{Corollary}

\newtheorem{ex}[thm]{Example}
\newtheorem{rmk}[thm]{Remark}

\newtheorem{const}[thm]{Construction}

\usepackage{times}
\usepackage{hyperref}
\def\Aut{{\rm Aut}}
\def\Hom{{\rm Hom}}
\def\ord{{\rm ord}}
\def\min{{\rm min}}
\def\max{{\rm max}}
\def\inf{{\rm inf}}
\def\sup{{\rm sup}}
\def\lim{{\rm lim}}

\def\Supp{{\rm Supp}}
\def\gr{{\rm gr}}
\def\pr{{\rm pr}}
\def\id{{\rm id}}
\def\wt{{\rm wt}}

\def\Diff{{\rm Diff}}

\def\triv{{\rm triv}}

\def\vol{{\rm vol}}
\def\Val{{\rm Val}}

\def\DH{{\rm DH}}

\def\Fut{{\rm Fut}}

\def\FL{{\rm FL}}

\def\SL{{\rm SL}}

\def\red{{\rm red}}

\def\QM{{\rm QM}}
\def\LC{{\rm LC}}
\def\Bl{{\rm Bl}}

\def\BF{\mathbf{F}}
\def\BP{\mathbf{P}}

\def\BV{\mathbf{V}}

\def\BH{\mathbf{H}}

\def\tl{\tilde{l}}
\def\tE{\tilde{E}}

\def\tQ{\tilde{Q}}
\def\tH{\tilde{H}}
\def\tX{\tilde{X}}
\def\tY{\tilde{Y}}

\newcommand{\IA}{{\mathbb A}}

\newcommand{\IC}{{\mathbb C}}

\newcommand{\IG}{{\mathbb G}}

\newcommand{\Ik}{{\mathbbm k}}

\newcommand{\IN}{{\mathbb N}}
 
\newcommand{\IP}{{\mathbb P}} 
\newcommand{\IQ}{{\mathbb Q}} 
\newcommand{\IR}{{\mathbb R}}

\newcommand{\IT}{{\mathbb T}}

\newcommand{\IZ}{{\mathbb Z}}

\newcommand{\CF}{{\mathcal F}}

\newcommand{\CL}{{\mathcal L}}

\newcommand{\CO}{{\mathcal O}}

\newcommand{\CX}{{\mathcal X}}

\newcommand{\bCX}{{\overline{\mathcal{X}}}}

\newcommand{\bCL}{{\overline{\mathcal{L}}}}

\newcommand{\seq}{\subseteq}
\newcommand{\la}{\langle}
\newcommand{\ra}{\rangle}

\newcommand{\bu}{\bullet}
\newcommand{\lam}{\lambda}

\newcommand{\D}{\Delta}

\title{A valuative criterion of K-polystability}

\author{Linsheng Wang}
\address{Shanghai Center for Mathematical Sciences, Fudan University, Shanghai, 200438, China}
\address{Department of Mathematics, Nanjing University, Nanjing, 210008, China}
\curraddr{}
\email{linsheng\_wang@outlook.com}
\thanks{}
\date{}
\dedicatory{}

\begin{document}

\begin{abstract}
For any log Fano pair with a torus action, we associate a computable invariant to it, such that the pair is (weighted) K-polystable if and only if this invariant is greater than one. As an application, we present examples of Fano varieties admitting $g$-solitons for any weight function $g$. 
\end{abstract}

\maketitle

\section{Introduction}

K-stability, first introduced by Tian \cite{Tia97} and later reformulated algebraically by Donaldson \cite{Don02}, is an algebraic condition detecting the existence of K\"ahler-Einstein metrics on Fano manifolds. Besides, uniform K-stability (\cite{FO16,Li19}) is a stronger stability notion that is determined by some numerical invariant (called $\delta$-invariant) of a Fano manifold. 
Let $(X,\D)$ be a log Fano pair. According to \cite[Theorem B]{BJ20}, it is uniformly K-stable if and only if $\delta(X,\D)>1$. In this case, it is K-stable. Additionally, by \cite[Remark A.3]{XZ19}, the log Fano pair is reduced uniformly K-stable if and only if $\delta^\red_{\IT}(X,\D)>1$ for some maximal torus $\IT\seq \Aut(X,\D)$. In this case, it is K-polystable. By the celebrated work of \cite[Theorem 1.6]{LXZ22}, uniform K-stability is equivalent to K-stability, and reduced uniform K-stability is equivalent to K-polystability. 
Therefore, in the study of the K-stability problem of a specific log Fano pair $(X,\D)$, it is essential to compute the invariant $\delta(X,\D)$ or $\delta^\red_{\IT}(X,\D)$. 

In general, computing $\delta(X,\D)$ can be challenging.  Nevertheless, the Abban-Zhuang estimate \cite[Theorem 3.2]{AZ22} has enabled the demonstration of $\delta(X,\D)>1$ in numerous cases, see for example \cite{ACC+}. However, to the best of the author's knowledge, there is no explicit result concerning the computation or estimation of the reduced delta invariant $\delta^\red_{\IT}(X,\D)$. In order to test K-polystability, Tian's alpha invariant \cite{Tia87} and equivariant K-stability \cite{Zhu21} are very useful. 

Instead of the reduced delta invariant, we will present a computable invariant that characterizes K-polystability, which stems from the Abban-Zhuang estimate \cite[Theorem 3.2]{AZ22}. 
Let $(X,\D)$ be a log Fano pair and $\IT\seq \Aut(X,\D)$ be a subtorus of rank $r$. Any one-parameter subgroup $\xi \in N=\Hom(\IG_m,\IT)$ determines a toric divisor $E_\xi$ over $X$. We may choose a basis $\{\xi_1,\cdots, \xi_r\}$ of the lattice $N$. 
Then taking adjunctions and refinements by the toric divisors $E_{\xi_1},\cdots, E_{\xi_r}$ succeedingly (see Construction \ref{Construction: toric divisor sequence} for details), we get a Fano type klt pair $(X_r,\D_r)$ and a $\IN\times \IN^r$-graded linear series $W^{X_r}_\bu$ on it. 
\begin{thm}[Theorem \ref{Theorem: main thm}]
\label{Theorem: Intro. main thm}
If $\IT\seq \Aut(X,\D)$ is a maximal subtorus and $\Fut|_N=0$, then the log Fano pair $(X,\D)$ is K-polystable (K-semistable) if and only if 
$$\delta(X_r, \D_{X_r}; W^{X_r}_\bu)> (\ge) 1. $$
\end{thm}

We denote by $R_\bu = R(X,\D)$ the anti-canonical ring of $(X,\D)$ and by $\BP\seq M_\IR$ the moment polytope of the $\IT$-action on $R_\bu$. The assumption $\Fut|_N=0$ implies that $A_{X,\D}(X_1) = S(R_\bu; X_1)$ for any $X_1:=E_\xi,\, \xi\in N$. Hence if we use the equivariant Abban-Zhuang estimate \cite[Theorem 4.6]{MW23} by refining to $X_1$, we will get an inequality like 
\begin{eqnarray*}
\delta_{p,\IT}(X,\D;R_\bu) \ge 
\min\Big\{1, \mathop{\inf}_{p_1\mapsto p} \delta_{p_1,\IT}(X_1,\D_{X_1};W^{X_1}_\bu)\Big\}, 
\end{eqnarray*} 
for any $\IT$-invariant closed point $p\in C_X(X_1)$, where the infimum runs over all the $\IT$-invariant closed points $p_1\in X_1$ mapping to $p\in X$. The dimension of the $\IT$-action decreases by one upon restriction to $X_1$, and the restriction of other toric divisors to $X_1$ remains toric. Therefore, we can successively refine by toric divisors and eventually obtain
\begin{eqnarray*}
\delta_{p,\IT}(X,\D;R_\bu) \ge 
\min\Big\{1, \mathop{\inf}_{p_r\mapsto p} \delta_{p_r}(X_r,\D_{X_r};W^{X_r}_\bu)\Big\}, 
\end{eqnarray*} 
where the $\IT$-action on $X_r$ is trivial and the infimum runs over all the closed points $p_r\in X_r$ mapping to $p\in X$. If for any $\IT$-invariant closed point $p\in X$ we can find such $X_r$ whose center on $X$ containing $p$ such that $\delta_{p_r}(X_r,\D_{X_r};W^{X_r}_\bu) \ge 1$ for any closed point $p_r\in X_r$ mapping to $p$, then $(X,\D)$ is K-semistable. Moreover, in many examples, we have a stronger condition 
$$\delta_{p_r}(X_r,\D_{X_r};W^{X_r}_\bu) > 1, $$
and it's natural to inquire whether we could extract further insights from it. Theorem \ref{Theorem: Intro. main thm} states that this leads to $(X,\D)$ being K-polystable.

\begin{rmk}\rm 
Furthermore, Theorem \ref{Theorem: Intro. main thm} offers the advantage that, to test K-(semi/poly)stability of a log Fano pair, it suffices to calculate the local delta invariant at points solely on one subvariety $C_X(X_r) \subseteq X$. There is no necessity to consider other $\IT$-invariant points. 
\end{rmk}
\begin{rmk}\rm 
Alternate forms of Theorem \ref{Theorem: Intro. main thm} have been used in some explicit examples \cite{MW23,MW24,LW24}. We also remark that an analogous statement of Theorem \ref{Theorem: Intro. main thm} for blow-up of projective bundles with $\IG_m$-action is proved by \cite{Mal24} independently.
\end{rmk}

Theorem \ref{Theorem: Intro. main thm} finds some of its inspiration from the following result about the sharpness of Abban-Zhuang estimate (see \cite[Theorem 3.2]{AZ22} or \cite[Theorem 4.6]{MW23}). 
Let $f:(X, \D)\to U$ be a $\IT$-equivariant projective morphism, where $U$ is a quasi-projective variety and $(X,\D)$ is a klt pair both admitting $\IT$-actions. 
We fix a $\IT$-invariant subvariety $Z\seq X$. 
\begin{thm}[Theorem \ref{Theorem: Sharpness of AZ}]
\label{Theorem: intro. sharpness AZ}
Let $F$ be a $\IT$-invariant plt-type divisor over $X$ with $C_X(F)\supseteq Z$. We denote by $\pi: Y \to X$ the associated plt-type blowup, and by $\D_F = \Diff_{F}(\D_Y)$, $\D_Y=\pi_*^{-1}\D$. Then for any $\IT$-invariant multi-graded linear series $V_\bu$ on $X$ with $F$-refinement $W_{\bu}$, we have 
\begin{eqnarray}
\label{Inequality: intro. Abban-Zhuang estimate}
\delta_{Z,\IT}(X,\D;V_\bu) \ge 
\min\Big\{\frac{A_{X,\D}(F)}{S(V_\bu;F)}, 
\mathop{\inf}_{Z'\mapsto Z} \delta_{Z',\IT}(F,\D_F;W_\bu)\Big\}, 
\end{eqnarray}
where the infimum runs over all $\IT$-invariant closed subvarieties $Z'\seq F$ 
(with $\dim Z' = \dim Z$) 
mapping to $Z\seq X$. 
Moreover, if $\ord_F$ is a minimizer of $\delta_{Z,\IT}(X,\D;V_\bu)$, then we have 
\begin{eqnarray}
\label{Inequality: intro. sharpness AZ}
\frac{A_{X,\D}(F)}{S(V_\bu;F)} =
\delta_{Z,\IT}(X,\D;V_\bu) \le 
\mathop{\inf}_{Z'\mapsto Z} \delta_{Z',\IT}(F,\D_F;W_\bu). 
\end{eqnarray}
\end{thm}

\begin{rmk} \rm
In general, if $\ord_F$ does not minimize $\delta_Z(X,\D;V_\bu)$, then we have
\begin{eqnarray*}
\frac{A_{X,\D}(F)}{S(V_\bu;F)} >
\delta_Z(X,\D;V_\bu) \ge
\mathop{\inf}_{Z'\mapsto Z} \delta_{Z',\IT}(F,\D_F;W_\bu). 
\end{eqnarray*}
But if $\ord_F$ minimizes $\delta_Z(X,\D;V_\bu)$, we do not know which term is larger on the right-hand side of (\ref{Inequality: intro. Abban-Zhuang estimate}). It seems possible that $A_{X,\D}(F)/S(V_\bu;F) > \mathop{\inf}_{Z'\mapsto Z} \delta_{Z',\IT}(F,\D_F;W_\bu)$. Hence we would fall short of the optimal bound if relying solely on this estimate. By (\ref{Inequality: intro. sharpness AZ}) we see that this scenario is precluded. 
\end{rmk}

Let's return to the case of Theorem \ref{Theorem: Intro. main thm}, but we do not need to assume that the subtorus $\IT=\IG_m^r\seq \Aut(X,\D)$ is maximal in the following. 
\begin{thm}[Theorem \ref{Theorem: main thm of almost complete}]
\label{Theorem: Intro. main thm. almost complete}
Assuming that $\Fut|_N=0$ and the multi-graded linear series $W^{X_r}_\bu$ is almost complete (see \cite[Definition 2.16]{AZ22} or (\ref{Eqnarray: almost complete assumption})) with respect to some big divisor $L$. Then there exists an effective $\IR$-divisor $F$ on $X_r$ such that $(X,\D)$ is K-polystable (K-semistable) if and only if $(X_r, \D_r+F)$ is K-stable or K-polystable (K-semistable). 
\end{thm}

For example, let $X$ be a Fano $\IT$-variety of complexity one (maximal torus of $\Aut(X)$ is of rank $r=\dim X-1$). Then $X_r\cong \IP^1$ is of Picard number one. In particular, $W^{X_r}_\bu$ is almost complete with respect to $\CO_{\IP^1}(1)$. 

In K-polystability problems, the assumption $\Fut|_N=0$ is always needed but may not hold in general. This is the case when a log Fano pair is destabilized by some product test configurations. Hence it is K-unstable and we could say nothing more. However, it is still possible that the log Fano pair is K-semi/polystable in the $g$-weighted setting. There are corresponding canonical metrics called $g$-solitons (see \cite{HL23}), and good moduli theory (claimed by \cite{BLXZ23}). 

\begin{rmk}\rm
Let $\BP\in M_\IR$ be the moment polytope of a log Fano pair with a $\IT$-action. 
We say that a continuous function $g:\BP\to \IR_{>0}$ is a {\it weight function} if the modified Futaki invariant $\Fut_g|_N = 0$ (see \cite{TZ02} or (\ref{Eqnarray: Fut_g = 0})). 
It is more natural to consider K-polystability problems in the $g$-weighted setting. 
All the above theorems, that is, Theorem \ref{Theorem: Intro. main thm}, \ref{Theorem: intro. sharpness AZ} and \ref{Theorem: Intro. main thm. almost complete} can be generalized to the $g$-weighted setting with no difficulty. And we state the theorems in the non-weighted form for simplicity of notations. 
\end{rmk}

The $g$-weighted version of Theorem \ref{Theorem: Intro. main thm. almost complete} is very useful. As an application, we generalize some existence results of K\"ahler-Ricci solitons on Fano threefolds in \cite{MW23,MW24} to the existence of $g$-solitons for arbitrary continuous weight function $g:\BP\to \IR_{>0}$ by \cite{HL23,BLXZ23}. 

\begin{thm}
\label{Intro. Theorem. 2.28, 3.14, 2.23(a)0 are weight-insensitive}
Let $X$ be a Fano threefold in the family №2.28 or №3.14 of Mori-Mukai's list, or $X=X_0$ is the optimal degeneration of a Fano threefold in family №2.23(a) (which is denoted by family №2.23(a0)), then $X$ is $g$-weighted K-polystable for any weight function $g:\BP\to \IR_{>0}$. In particular, it admits a $g$-soliton if $\Ik=\IC$. 
\end{thm}

The proof follows from Example \ref{Example: 2.28 and 3.14}. As a consequence, we obtain examples of $g$-weighted moduli spaces, which are isomorphic to GIT-moduli spaces of cubic curves $C\seq \IP^2$ or biconic curves $C\seq \IP^1\times \IP^1$, see \cite[Theorem 1.3]{MW23} and \cite[Theorem 1.5]{MW24}. 

\begin{rmk}\rm
Since this paper appeared on the arXiv, Thibaut Delcroix \cite{Del24} found more examples of Fano threefolds having the above property, and said that they are {\it weight-insensitive}. Precisely, a log Fano pair $(X,\D)$ with a torus $\IT$-action is called $\IT$-equivariantly weight-insensitive K-polystable if, for any weight function $g$ on $\BP(\IT)$, $(X,\D)$ is $\IT$-equivariantly $g$-weighted K-polystable. With this terminology, Theorem \ref{Intro. Theorem. 2.28, 3.14, 2.23(a)0 are weight-insensitive} says that Fano threefolds in family №2.28, №3.14 and Fano varieties in family №2.23(a0) are weight-insensitive K-polystable. 

Moreover, these examples of Fano varieties have non-trivial moduli: for any weight function $g$, the $g$-weighted K-moduli space of Fano threefolds in family №2.28 (or №3.14) is isomorphic to the GIT-moduli space of plane cubic curves by Example \ref{Example: 2.28 and 3.14}; the $g$-weighted K-moduli space of Fano varieties in family №2.23(a0) is isomorphic to the GIT-moduli space of biconic curves in $\IP^1\times\IP^1$ by Example \ref{Example: 2.28 and 3.14}. 

The optimal degeneration theory of Fano varieties was generalized in \cite{Wang24} (see \cite[Definition 3.1]{Wang24}) by modifying the exponential function in the definition of $\BH$-invariants (see \cite[Section 2.5]{HL20} or \cite[Section 2.2.3]{BLXZ23}). By the weight insensitivity of Fano varieties in family №2.23(a0), for any Fano threefolds $X$ in family №2.23(a), its {\it generalized optimal degeneration} (see \cite[Theorem 1.1]{Wang24}) is always $X_0$ (the optimal degeneration of $X$). 
\end{rmk}

The paper is organized as follows. In Section \ref{Section: Preliminaries} we collect some basic results about Fano $\IT$-varieties. We prove the sharpness of Abban-Zhuang estimate in Section \ref{Section: The sharpness of Abban-Zhuang estimate}. In Section \ref{Section: A Valuative Criterion of (Weighted) K-polystability}, we establish the main theorem of the paper, that is, a valuative criterion of K-polystability. 
In Section \ref{Section: Applications: almost complete linear series}, we consider the problem under the almost complete condition and show the existence of $g$-solitons on some Fano varieties as applications. We will give a another construction of the triple $(X_r, \D_{X_r}; W^{X_r}_\bu)$ in Section \ref{Section: Qdlt Fano type models}.

\noindent {\bf Acknowledgments}. The author would like to express deep gratitude to his advisor, Gang Tian, for his constant support and guidance. The author want to thank Chenyang Xu, Yuchen Liu and Ziquan Zhuang for suggestions about Fano varieties with torus action. 
The author also thank Xin Fu, Jiyuan Han, Xiaowei Jiang, Minghao Miao, Kewei Zhang and Shengxuan Zhou for helpful discussions. The author was partially supported by the NKRD Program of China (\#2020YFA0713200), (\#2023YFA1010600) and LNMS.

\section{Preliminaries}
\label{Section: Preliminaries}

\subsection{Notations and Conventions}
We work over an algebraically closed field $\Ik$ of characteristic $0$. A {\it variety} is a separated integral scheme of finite type over $\Ik$. A {\it pair} $(X, \Delta)$ consists of a normal variety $X$ and an effective $\IQ$-divisor $\Delta$ on $X$ such that $K_X+\Delta$ is $\IQ$-Cartier, and it is a {\it subpair} if not assuming $\D$ is effective. 
A pair $(X, \Delta)$ is called {\it log Fano} if it is klt, $X$ is projective, and $-K_X-\Delta$ is ample. 
A polarized pair $(X,\D;L)$ consists of a projective pair $(X,\D)$ and a $\IQ$-Cartier ample divisor $L$ on $X$. Fix an integer $l_0 > 0$ such that $l_0 L$ is Cartier. We denote by $R:=R(X;L):=\oplus_{m\in l_0\IN} R_m$ the section ring of $L$ where $R_m := H^0(X, mL)$. If $(X,\D)$ is log Fano, we will always choose $L=-K_X-\D$. In this case $R(X,\D):=R(X;-K_X-\D)$ is the anti-canonical ring of $(X,\D)$. 

Let $G$ be an algebraic group. For any $G$-action $\eta: G\times X\to X$, we denote by $X^G\seq X$ the subscheme of $G$-invariant points. A subvariety $Y\seq X$ is called {\it $G$-invariant} if $\eta^{-1}(Y)=G\times Y$. The action $\eta$ is called {\it effective} if it is not induced from other group action, that is, if there is a morphism of algebraic groups $\phi: G\to G'$ such that $\eta=\eta'\circ(\phi\times \id)$ for some $G'$-action $\eta': G'\times X\to X$, then $\phi$ is an isomorphism. 

A {\it valuation} $v$ on $X$ is an $\IR$-valuation on the rational function field $K(X)$ with a center $c_X(v)$ on $X$ and $v|_{\Ik^\times}=0$. We denote by $C_X(v)=\overline{c_X(v)}\subseteq X$ the corresponding closed irreducible subscheme on $X$. Let $\D$ be an effective $\IQ$-divisor on $X$ such that $K_X+\D$ is $\IQ$-Cartier. The {\it log discrepancy} $A_{X,\D}(v)$ of a valuation $v$ on $X$ is defined in \cite{jm12}. We denote by $\Val_X$ the set of valuations on $X$ and by $\Val_{X}^\circ$ the subset of non-trivial valuations $v\in\Val_{X}$ with $A_{X,\D}(v)<+\infty$ of $v$ for some $\IQ$-divisor $\D$ such that $K_X+\D$ is $\IQ$-Cartier (it does not depend on the choice of $\D$). . For any projective morphism $X\to U$ between normal varieties and irreducible subvariety $Z\seq U$, we denote by $\Val_{X,Z}\seq \Val_X$ be the subset of valuations $v$ with $C_X(v)\supseteq Z$. If $X$ admits a torus $\IT=\IG_m^r$-action, we denote by $\Val_X^{\IT}\seq \Val_X$ the subset of $\IT$-invariant valuations on $X$. 

We say that a variety $E$ is a prime divisor {\it over} $X$ if there exists a projective birational morphism $\pi:Y\to X$ such that $E\seq Y$ is a prime divisor on $Y$. If we assume moreover that $Y$ is normal and $-E$ is ample $\IQ$-Cartier, then $E$ is called {\it primitive} over $X$, and $\pi$ is called the associated prime blowup. For any $\IQ$-divisor $D, D'$ on $X$, we denote by $D\vee D'$ the smallest $\IQ$-divisor $D''$ such that $D''\ge D$ and $D'$.
Let $(X,\D)$ be a pair and $E$ be a primitive divisor over $X$. 
Then $E$ is said to be of {\it plt-type} over $(X,\D)$ if $(Y,E\vee\pi_*^{-1}\D)$ is plt, where $\pi: Y\to X$ is the prime blowup of $E$.

\subsection{Special test configurations and special divisorial valuations}
\begin{defi} \rm
Let $(X,\D;L)$ be a polarized pair. A {\it test configuration (TC)} of $(X,\D,L)$ is a collection $(\CX, \D_\CX,\CL,\eta)$ consisting of
\begin{itemize}
\item A variety $\CX$ with a $\IG_m$-action generated by a holomorphic vector field $\eta\in \Hom(\IG_m, \Aut(\CX))$; 
\item A $\IG_m$-equivariant morphism $\pi: \CX\to \IA^1$, where the $\IG_m$-action on $\IA^1$ is standard; 
\item A $\IG_m$-equivariant $\pi$-semiample $\IQ$-Cartier divisor $\CL$ on $\CX$; 
\item A $\IG_m$-equivariant trivialization over the punctured plane 
$$i_\eta:(\CX,\CL)|_{\pi^{-1}((\IA^1\setminus\{0\}))}\cong (X,L)\times (\IA^1\setminus\{0\}),$$ 
which is compatible with $\pi$ and $\pr_1$. And $\D_\CX$ is the closure of $i_\eta^{-1}(\D\times(\IA^1\setminus\{0\}))$ in $\CX$.  
\end{itemize}
\end{defi}

If $\CX$ is a normal variety, then $(\CX,\D_\CX,\CL,\eta)$ is called a {\it normal test configuration}. 
In the log Fano case, we always choose $L=-K_X-\D$, and a normal test configuration $(\CX,\D_\CX,\CL,\eta)$ is called {\it weakly special} (resp. {\it special}) if $(\CX,\CX_0+\D_\CX)$ is lc (resp. plt), and $\CL=-K_{\CX/\IA^1}-\D_\CX + c\CX_0$ for some $c\in\IQ$. Note by adjunction that $(\CX, \D_\CX,\CL)$ being special is equivalent that the central fiber $(\CX_0, \D_{\CX,0})$ is log Fano. 
A normal test configuration $(\CX,\D_\CX)$ is called of {\it product type} if $(\CX,\D_\CX)$ is $\IG_m$-equivariantly isomorphic to $(X,\D)\times \IA^1$. 

Since $(\CX,\CL)|_{\CX\setminus \CX_0}$ is trivial over the punctured plane, we naturally have a compatification $\overline{\pi}:(\bCX, \bCL) \to \IP^1$ of $\pi: (\CX,\CL) \to \IA^1$ by adding a trivial fiber $X_\infty \cong X$ over $\infty \in \IP^1$.

\begin{defi}[K-stability] \rm 
Let $(X,\D)$ be a log Fano pair. For any normal test configuration $(\CX,\D_\CX,\CL)$ of $(X,\D)$, the {\it generalized Futaki invariant} is defined by 
\begin{eqnarray*}
\Fut(\CX,\D_\CX;\CL)
:=
\frac{1}{(n+1)(-K_X-\D)^{n}}
\Big(
n \overline{\CL}^{n+1} 
+ (n+1)(K_{\overline{\CX}/\IP^1} 
+ \D_{\overline{\CX}})\cdot \overline{\CL}^n
\Big). 
\end{eqnarray*}
The log Fano pair $(X, \D)$ is called {\it K-stable (K-semistable)} if $\Fut(\CX, \D_\CX,\CL)> (\ge)\, 0$ for any normal test configuration $(\CX, \D_\CX, \CL)$ of it; it is called {\it K-polystable} if it is K-semistable and any normal test configuration $(\CX, \D_\CX, \CL)$ with $\Fut(\CX,\D_\CX;\CL)=0$ is of product type. 
\end{defi}

By \cite{LX14}, we could replace ``normal test configuration'' with ``special test configuration'' in the definition of K-(semi/poly)stability of log Fano pairs. 

\begin{rmk}\rm
\label{Remark: K-stable implies no G_m action}
A log Fano pair $(X,\D)$ is K-stable implying that it admits no $\IG_m$-action. Indeed, let $\rho:\IG_m\to X$ be a one-parameter subgroup. Then the sum of $\Fut$ of the product test configurations induced by $\rho$ and $\rho^{-1}$ is $0$, see (\ref{Eqnarray: Fut on N}) for details. 
\end{rmk}

\begin{defi} \rm 
A prime divisor $E$ over $(X,\D)$ is called {\it (weakly) special} if there exists a special test configuration (weakly special test configuration with integral central fiber) $(\CX,\D_\CX)$ such that $\ord_{\CX_0}|_{\CX_1} = c\cdot \ord_E$ for some $c\in \IQ_{>0}$.  
\end{defi}

We have the following characterization of (weakly) special divisors. 

\begin{thm} \cite[Theorem 4.24]{Xu24}
\label{Theorem: weakly special divisor definition}
A prime divisor $E$ over $(X,\D)$ is weakly special if and only if there exists a $\IQ$-complement $D$ of $(X,\D)$ such that $E$ is an lc place of $(X,\D+D)$. 
\end{thm}

\begin{thm}\cite[Theorem 4.27]{Xu24}
\label{Theorem: special divisor definition}
A prime divisor $E$ over $(X,\D)$ is special if and only if there exists a plt-Fano type model $\pi: (Y, E)\to (X,\D)$, that is, there exists a birational morphism $\pi:Y\to X$ and an effective $\IQ$-divisor $D$ on $Y$ such that $(Y,D+E)$ is plt, $D+E\ge \pi_*^{-1}\D$, and $-(K_Y+D+E)$ is ample. 
\end{thm}

\subsection{Valuative criterion of K-stability}

\begin{defi}\rm 
\label{Definition: filtrations}
Let $(X,\D;L)$ be a polarized pair. 
A {\it filtration} $\CF$ on $R=R(X,\D;L)$ is a collection of subspaces $\CF^\lam R_m \subseteq R_m$ for each $\lam \in \IR$ and $m\ge 0$ such that
\begin{itemize}
\item {\it Decreasing.} $\CF^\lam R_m \supseteq \CF^{\lam'}R_m $ for  $\lam \le \lam'$; 
\item {\it Left-continuous.} $\CF^\lam R_m=\CF^{\lam-\epsilon}R_m$ for $0<\epsilon \ll 1$; 
\item {\it Bounded.} $\CF^\lam R_m = R_m$ for $\lam \ll 0$ and $\CF^\lam R_m = 0$ for $\lam \gg 0$; 
\item {\it Multiplicative.} $\CF^\lam R_m \cdot \CF^{\lam'}R_{m'} \subseteq \CF^{\lam+\lam'}R_{m+m'}$. 
\end{itemize}
Since $R$ is finitely generated and $\CF$ is bounded and multiplicable, there is a constant $C>0$ such that $\CF^{-mC}R_m=R_m$ for all $m$. A filtration $\CF$ is called {\it linearly bounded} if there is a constant $C>0$ such that $\CF^{mC}R_m=0$ for all $m$. We will always assume that filtration is linearly bounded in this paper. 
\end{defi}

\begin{rmk} \rm
\label{Remark: valuation and TC to filtration}
For any valuation $v$ on $X$, there is a filtration $\CF_v$ on $R$ defined by
$$\CF_v^\lam R_m := \{s\in R_m\mid v(s)\ge \lam\}. $$
If $A_X(v)<+\infty$, then $\CF_v$ is linearly bounded, see \cite{BJ20}. In particular, the trivial valuation induces the trivial filtration
$\CF_{\triv}^0 R_m = R_m,\,\,\CF_{\triv}^{>0}R_m= 0. $

For any test configuration $(\CX,\D_\CX,\CL)$ of $(X,\D;L)$, we have the following
$\IZ$-filtration $\CF=\CF_{(X,\D_\CX;\CL)}$ on $R=R(X;L)$
\begin{eqnarray*}
\CF^\lam R_m
&:=&\{f\in H^0(X,mL)\mid t^{-\lam} \bar{f} \in H^0(\CX,m\CL)\}, 
\end{eqnarray*}
where $t$ is the parameter on $\IA^1$, and $\bar{f}$ is the $\IG_m$-extension of $f$ on $\CX\setminus\CX_0$ and viewed as a rational section of $m\CL$. 
\end{rmk}

\begin{defi}\rm
Let $\CF$ be a linearly bounded filtration on $R$ and $m\in l_0\IN$. For any $s\in R_m$, we set $\ord_\CF(s)=\max\{\lam: s\in\CF^{\lam}R_m\}$. For any basis $\{s_i\}$ of $R_m$, the divisor 
$$D=\frac{1}{m\cdot \dim R_m} \sum_i \{s_i=0\}$$
is called an {\it $m$-basis type divisor} of $R_\bu$. 
A basis $\{s_i\}$ (or the correspondence $m$-basis type divisor $D$) of $R_m$ is called {\it compatible} with $\CF$ if $\CF^\lam R_m$ is generated by $\{s_i: \ord_\CF(s_i)\ge \lam \}$ for any $\lam\in \IR$. It's not difficult to see that $\ord_\CF(D)$ achieves the maximum for any $m$-basis type divisor $D$ if and only if $D=D_c$ is compatible with $\CF$. We define 
\begin{eqnarray*} 
\lam^{(m)}_\max(\CF) 
&:=& \max\{\lam\in \IR\mid \CF^{\lam}R_m \ne 0\}, \\
S_m(\CF) 
&:=&
\mathop{\sup}_{D:\, m\text{-basis type}} v(D)  
\,\,\,=\,\,\,
v(D_c)
\,\,\,=\,\,\,
\sum_{\lam}\frac{\lam}{m} \cdot 
\frac{\dim \gr_\CF^{\lam} R_m}{\dim R_m}, \qquad
\end{eqnarray*} 
where $D_c$ is compatible with $\CF$. 
By \cite{BJ20}, the limits exist as $m\to \infty$ and we define
\begin{eqnarray*} 
\lam_\max(\CF) 
&:=& \mathop{\sup}_{m\in \IN} \frac{\lam^{(m)}_\max}{m} 
= \mathop{\lim}_{m\rightarrow \infty} \frac{\lam^{(m)}_\max}{m}, \\
S(\CF) &:=& \mathop{\lim}_{m\to \infty} S_m(\CF). 
\end{eqnarray*}
The invariant $S(\CF)$ is called the {\it expected vanishing order} of $\CF$. We always denote by $S(v)=S(\CF_v)$ and $S(E)=S(\CF_{\ord_E})$ for any valuation $v$ and prime divisor $E$ over $X$. For any non-trivial valuation $v$ on $X$, the {\it Fujita-Li invariant} is defined by 
\begin{eqnarray*} 
\FL(v) := A_{X,\D}(v) - S(v). 
\end{eqnarray*}
\end{defi}

\begin{thm}[Fujita-Li's valuative criterion]
A log Fano pair $(X,\D)$ is K-stable if and only if 
$\FL(v)>0$ for any valuation $v$ over $X$. 
\end{thm}

\begin{defi} \rm
\label{Eqnarray: definition of delta}
The {\it delta invariant} of a log Fano pair $(X, \D)$ is defined by 
\begin{eqnarray}
\label{Eqnarray: definition of delta}
\delta(X,\D) \,\,\,=\,\,\, 
\mathop{\inf}_{v\in \Val^\circ_X} \frac{A_{X,\D}(v)}{S(v)}, 
\end{eqnarray}
where $\Val_X^\circ \seq \Val_X$ is the subset of non-trivial valuations satisfying $A_{X,\D}(v)<+\infty$. 
\end{defi}

\begin{thm}\cite{BJ20,LXZ22}
A log Fano pair $(X,\D)$ is K-stable if and only if 
$\delta(X,\D)>1$. 
\end{thm}

\begin{rmk}\rm
If $(X,\D)$ admits a torus $\IT$-action, we may define the $\IT$-equivariant delta invariant $\delta_{\IT}(X,\D)$ by taking infimum for $v\in \Val^\circ_X$ in (\ref{Eqnarray: definition of delta}). By $\IT$-equivariant K-stability \cite{Zhu21} and Remark \ref{Remark: K-stable implies no G_m action}, we have $\delta_{\IT}(X,\D)\le 1$. 
\end{rmk}

\subsection{Toric divisors over Fano $\IT$-varieties}
Let $(X,\D)$ be a log Fano pair with a $\IT=\IG_m^r$-action. 
We denote by $N=\Hom(\IG_m,\IT) \cong \IZ^r$ the coweight lattice of the $\IT$-action. Then any $\xi \in N$ determines a one-parameter subgroup $\xi: \IG_m \to \IT \seq \Aut(X,\D), t\mapsto \xi_t$. The $\IT$-action on $\CO_X$ is given by 
\begin{eqnarray} 
\label{Eqnarray. 1-PS action}
(\xi_t^*f)(x) = f(\xi_{t^{-1}}(x)), 
\end{eqnarray}
for any $\xi\in N, t\in \IG_m, f\in \CO_X$ and $x\in X$, see \cite[(2.21)]{Xu24}. 
Since the anti-canonical divisor $-(K_X+\D)$ admits a canonical $\IT$-linearization, we see that $R_m = H^0(X,-m(K_X+\D))$ admits a canonical weight decomposition $R_m=\oplus_{\alpha \in M} R_{m,\alpha}$, where $M=\Hom(\IT,\IG_m)\cong N^\vee$ is the weight lattice of the $\IT$-action and  
\begin{eqnarray} 
\label{Eqnarray. Weight decomposition components}
R_{m,\alpha} = \{s\in R_m\mid \xi_t^*s = t^{\la \alpha, \xi\ra} \cdot s \text{ for any } \xi \in N, t\in \IG_m\}. 
\end{eqnarray}

For any $\xi \in N$, it determines a product test configuration 
$$(\CX_\xi, \D_{\CX_\xi}, \CL_\xi,\eta=\eta_\xi):=(X, \D,-(K_X+\D), (\xi,1))\times \IA^1, m\in l_0\IN,$$
where $\eta=\eta_\xi = (\xi,1)$ is determined by the isomorphism
\begin{eqnarray*} 
i_\eta: \CX_\xi \setminus \CX_{\xi,0} &\cong& X\times (\IA^1\setminus\{0\}), \\
(x,t) &\mapsto& (\xi_{t^{-1}}(x), t).
\end{eqnarray*}
for any $(x,t)\in X\times (\IA^1\setminus\{0\})$. 
This is a special test configuration, hence determines a special divisorial valuation
\begin{eqnarray} 
\label{Eqnarray: toric divisor}
\wt_\xi := \ord_{\CX_0}|_{\CX_1} =: c_\xi \cdot \ord_{E_\xi}, 
\end{eqnarray}
where $c_\xi \in\IZ_{\ge 1}$ since the value group $\Gamma = \ord_{\CX_0}(K(\CX)^*)$ is $\IZ$ (see \cite[Proof of Theorem 4.6]{BHJ17}) and $E_\xi$ is a special divisor over $X$. 

\begin{defi}\rm 
We say that $E_\xi$ is a {\it toric divisor} over $X$ with respect to the $\IT$-action. 
\end{defi}

For any $s\in R_{m,\alpha}$, let $\pr_1^*s$ be its pull-back on $X\times (\IA^1\setminus\{0\})$. The $\IG_m$-invariant rational section $\bar{s} = i_\eta^*\pr_1^*s$ of $m\CL_\xi$ determined by $s$ is 
$$(i_{\eta}^*\pr_1^*s)(x,t)
=(\pr_1^*s)(i_{\eta}(x,t))
=(\pr_1^*s)(\xi_{t^{-1}}(x), t)
= s(\xi_{t^{-1}}(x)) 
= (\xi_t^*s)(x) 
= t^{\la\alpha,\xi\ra} s(x). $$
Hence $i_{\eta}^*(\pr_1^*s) = t^{\la\alpha, \xi\ra}\cdot (\pr_1^*s)$.  
So by Remark \ref{Remark: valuation and TC to filtration}, the filtration induced by the test configuration $\CX_\xi$ is 
\begin{eqnarray*} 
\CF_{\CX_\xi}^\lam R_m 
\,\,\,=\,\,\, 
\bigoplus_{\la\alpha,\xi\ra\ge \lam} R_{m,\alpha}
\,\,\,=\,\,\, 
\CF_{\triv,\xi}^\lam R_m. 
\end{eqnarray*}
On the other hand, by \cite[Lemma 6.6]{Li17}, \cite[Claim 5.4]{Fuj16} or \cite[Lemma 5.17]{BHJ17}, we have
\begin{eqnarray*} 
\CF_{\CX_\xi}^\lam R_m 
\,\,\,=\,\,\,
\{f\in R_m\mid \wt_\xi(f) \ge \lam+ mA_{X,\D}(\wt_\xi)\}
\,\,\,=\,\,\, 
\CF_{\wt_\xi}^{\lam+mA_{X,\D}(\wt_\xi)}R_m. 
\end{eqnarray*}
Hence $\CF_{\wt_\xi}=\CF_{\CX_\xi}(A_{X,\D}(\wt_\xi))$, and 
\begin{eqnarray} 
\label{Eqnarray: valuation of product TC}
\wt_\xi(s) = \la\alpha,\xi\ra + mA_{X,\D}(\wt_\xi), \quad 
s\in R_{m,\alpha}. 
\end{eqnarray}

Recall that the moment polytope $\BP\seq M_\IR$ of the $\IT$-action on $(X,\D)$ is defined by 
\begin{eqnarray*}
\BP   = \overline{\bigcup_{m\in l_0 \IN} \frac{1}{m} \BP_m}, \quad
\BP_m = \{\alpha \in M_\IZ\mid R_{m, \alpha} \ne 0 \}. 
\end{eqnarray*}
We have the following probability measures on $\BP$
\begin{eqnarray*} 
\DH_{\BP,m}
=
\sum_{\alpha\in \BP_m}  
\frac{\dim R_{m,\alpha}}{\dim R_m} \cdot
\delta_\frac{\alpha}{m}, \quad 
\DH_\BP = \mathop{\lim}_{m\to \infty} \DH_{\BP,m}. 
\end{eqnarray*}
Then the Futaki invariant for any $\xi \in N_\IR$ (product TC) can be written as
\begin{eqnarray} 
\label{Eqnarray: Fut on N}
\Fut(\xi) = -\int_\BP \la\alpha,\xi\ra \cdot \DH_\BP(d\alpha). 
\end{eqnarray}

\begin{rmk}\rm 
By (\ref{Eqnarray: valuation of product TC}) and (\ref{Eqnarray: Fut on N}), we have $\Fut(\xi) = \FL(\wt_\xi)$ for any $\xi \in N_\IQ$. Hence we directly see that if a log Fano pair $(X,\D)$ admits a $\IG_m$-action, then $\delta_\IT(X,\D) \le 1$. 
\end{rmk}

\subsection{The $\xi$-twist of valuations}
Let $X$ be a proper variety of dimension $n$ with an effective $\IT=\IG_m^r$-action. Then there exists a proper variety $Z$ of dimension $n-r$ and a $\IT$-equivariant birational map $\pi: X \dashrightarrow Z\times \IT$, where the $\IT$-action on $Z$ is trivial. The function field $K(X)$ of $X$ is the fractional field of $K(Z)[M] = \oplus_{\alpha \in M} K(Z) \cdot 1^\alpha$. For any valuation $\mu$ on $Z$ and $\xi \in N_\IR$ we define the $\IT$-invariant valuation $v_{\mu, \xi}$ on $X$ such that 
$$v_{\mu, \xi}(f)
= \min_\alpha\{\mu(f_\alpha)+\la \alpha, \xi \ra\},$$
for any $f=\sum_\alpha f_\alpha \cdot 1^\alpha \in K(Z)[M]$. By \cite[Lemma 4.2]{BHJ17} we know that any $\IT$-invariant valuation over $X$ is obtained in this way, and we get a non-canonical isomorphism 
$$\Val^\IT_X \cong \Val_Z \times N_\IR. $$ 
By considering toric divisors over $(X,\D)$, we will give an explicit construction of this isomorphism, see Lemma \ref{Lemma: plt Fano type model}. 
For any $v = v_{\mu, \xi_0}\in \Val^\IT_X$ and $\xi \in N_\IR$, we define the {\it $\xi$-twist} of $v$ by $v_\xi:= v_{\mu,\xi_0+\xi}$. One can check that the definition is independent of the choice of the birational map $X\dashrightarrow Z\times \IT$. 

\begin{thm}\cite[Proposition 3.12]{Li19} Let $(X,\D)$ be a log Fano pair admitting a $\IT$-action with co-weight lattice $N$. 
For any $\xi\in N_\IR$ and $v\in \Val^\IT_X$, we have
$$\Fut(\xi) = \FL(v_\xi) -\FL(v). $$
\end{thm}
We may define the function $\theta_\xi$ on $\Val^\IT_X$ by
\begin{eqnarray} 
\label{Eqnarray: A(v_xi) = A(v)+theta_xi(v)}
A_{X,\D}(v_\xi) = A_{X,\D}(v) + \theta_\xi(v). 
\end{eqnarray}
If $\Fut|_N=0$, then 
\begin{eqnarray} 
\label{Eqnarray: S(v_xi) = S(v)+theta_xi(v)}
S(v_\xi) = S(v) + \theta_\xi(v). 
\end{eqnarray}


The following lemma will be used in the proof of our main theorem. 

\begin{lem}
\label{Lemma: plt Fano type model}
Let $(X,\D)$ be a log Fano pair with an effective $\IG_m$-action, and $R_m=\oplus_{\alpha \in \IZ} R_{m, \alpha}$ be the canonical weight decomposition of the anti-canonical ring $R=\oplus_{m\in l_0\IN} R_m$. Then for any primitive $\xi\in N(\IG_m) \cong \IZ$, there exists a special divisor $E$ over $(X, \D)$ such that, for sufficiently divisible $m$, we have 
\begin{eqnarray}
\label{Eqnarray: ord_E = wt_xi + A(E)}
\ord_{E}(s) = \la\alpha,\xi\ra + m A_{X,\D}(E), 
\quad \forall \, s \in R_{m,\alpha}. 
\end{eqnarray}

Moreover, there exists an effective $\IQ$-divisor $\D_E$ on $E$ such that $(E,\D_E)$ is of klt Fano type, and an isomorphism $i:\Val_E\times N(\IG_m)_\IR \to \Val^{\IG_m}_X$ (still denote $i(v,0)$ by $v$) such that 
\begin{eqnarray}
\label{Eqnarray: A_X(v) = A_E(v)}
A_{X,\D}(v) = A_{E,\D_E}(v), \quad \forall v\in \Val_E. 
\end{eqnarray}
\end{lem}
\begin{proof}
Let $E$ be the toric divisor induced by $\xi$, then by (\ref{Eqnarray: toric divisor}) and (\ref{Eqnarray: valuation of product TC}) we have 
\begin{eqnarray*} 
\ord_E(s) := c_\xi^{-1} \cdot \la\alpha,\xi\ra + mA_{X,\D}(E), \quad
\forall s \in R_{m,\alpha},  
\end{eqnarray*}
for some $c_\xi \in \IZ_{\ge 1}$ and sufficiently divisible $m$ (in particular, $mA_{X,\D}(E)\in \IZ$). 
Since the $\IG_m$-action is effective and $\xi\in N$ is primitive, there exists $\alpha\in M(\IG_m)$ such that $\la\alpha,\xi\ra=1$. If $c_\xi >1$, then for any $0\ne s\in R_{m,\alpha}$, we have 
\begin{eqnarray*} 
\ord_E(s) = c_\xi^{-1} + mA_{X,\D}(E) \notin \IZ, 
\end{eqnarray*}
which contradicts that $\ord_E(K(X)^*)\seq \IZ$. Hence $c_\xi=1$. 

By the equivariant version of Theorem \ref{Theorem: special divisor definition}, there exists a $\IG_m$-equivariant birational morphism $\pi:Y \to X$ extract precisely the prime divisor $E$, and there exists an effective $\IQ$-divisor $D$ on $Y$ with $D\ge \pi_*^{-1}\D +E$ and $\lfloor D \rfloor=E$ such that $(Y,D)$ is plt and $-K_Y-D$ is ample. Hence by adjunction $(E,\Diff_E(D-E))$ is a log Fano pair. Let $\D_Y=\pi_*^{-1}\D$ and $\D_E=\Diff_E(\D_Y) \le \Diff_E(D-E)$. We see that $(E,\D_E)$ is of klt Fano type. 

Finally, we construct the isomorphism of the valuation spaces using Bialynicki-Birula decomposition \cite[Theorem 4.1]{BB73}. Since the $\IG_m$-action lifts to $Y$ and $E$ is toric with respect to this $\IG_m$-action, we have $E\seq Y^{\IG_m}$. In other words, the induced $\IG_m$-action on $E$ is trivial. 
Let $f:\tY \to Y$ be a $\IG_m$-equivariant log resolution of the plt pair $(Y, \D_Y+E)$, and $(\tY, \D_{\tY}+\tE)$ be the crepant pull-back. Then the restriction $f_E: \tE\to E$ of $f$ is a log resolution of $(E,\D_E)$, and we denote by $(\tE,\D_{\tE})$ the crepant pull-back. It's clear that 
$\D_{\tE}=\D_{\tY}|_{\tE}$. 

The $\IG_m$-action also lifts to $\tY$ and $\tE \seq \tY^{\IG_m}$. Since $\tE$ is of codimension one, by \cite[Theorem 4.1]{BB73} there exists a $\IG_m$-invariant open subset $\tY^+ \seq \tY$ containing $\tE$ and a $\IG_m$-equivariant morphism $\tau: \tY^+ \to E$, which is a locally trivial $\IA^1$-bundle. 
For any valuation $v\in \Val_E$, let $U\seq E$ be an affine open neighbourhood of the generic point $\eta$ of $C_E(v)$ in $E$ such that $\tau|_U: \tau^{-1}U \to U$ is a trivial $\IA^1$-bundle. Then $\CO_{\tau^{-1}U} \cong \CO_U[t]$ and there exists an isomorphism $\kappa: K(U)(t)\cong K(\tau^{-1}U) = K(\tY^+)$ induced by $\tau$. We extend $v:K(U)^*\to \IR$ trivially to $v: K(U)(t)^*\to \IR$ by letting $v(t)=0$ and denote by $\tau^*v = \kappa_*v$, which is a valuation centered on $\tau^{-1}U \seq \tY^+$. The definition of $\tau^*v$ is independent of the choice of $U$. We get an inclusion:  
\begin{eqnarray*} 
\Val_{\tE} \stackrel{\tau^*}{\longrightarrow} \Val^{\IG_m}_{\tY^+}, \quad v\mapsto \tau^*v. 
\end{eqnarray*}
Hence
\begin{eqnarray*} 
\Val_{E}=\Val_{\tE} \stackrel{\tau^*}{\longrightarrow} \Val^{\IG_m}_{\tY^+} \seq \Val^{\IG_m}_{\tY} = \Val^{\IG_m}_X. 
\end{eqnarray*}
Then by \cite[Lemma 4.2]{BHJ17}, we get an isomorphism $i:\Val_E\times N(\IG_m)_\IR \to \Val^{\IG_m}_X$. 

Let $\D_{\tY^+} = \D_{\tY}|_{\tY^+}$ and consider the vertical components of this $\IG_m$-invariant divisor. With the same argument in the previous paragraph, we see that each vertical component $W$ is a locally trivial $\IA^1$-bundle over $W\cap \tE$. Hence we have (omit $\tau^*$)
\begin{eqnarray*} 
A_{\tY^+, \D_{\tY^+}}(v) = A_{\tE,\D_{\tE}}(v), \quad 
\forall v \in \Val^\circ_{\tE}.  
\end{eqnarray*}
Since all the pull-backs are crepant, we conclude that 
\begin{eqnarray*} 
A_{X,\D}(v) = A_{Y,\D_Y+E}(v) = A_{\tY,\D_{\tY}+\tE}(v) = A_{\tY,\D_{\tY}}(v)  = A_{\tY^+, \D_{\tY^+}}(v) = A_{\tE,\D_{\tE}}(v) = A_{E,\D_E}(v),   
\end{eqnarray*}
for any $v\in\Val_{\tE}$, where the third equality follows from $C_{\tY}(v) \nsubseteq \tE$. 
\end{proof}

\subsection{Multi-graded linear series and refinements} 
\begin{defi}\rm
Let $(X, \D)$ be a klt pair, and $L,L_1, \cdots, L_l$ be a sequence of line bundles on $X$. A {\it $\IN\times \IN^l$-graded linear series} $V_\bu$ on $X$ associated to those $L_i$ is a collection of finite dimensional subspaces
\begin{eqnarray*}
V_{m,\beta} 
\seq H^0(X_l, mL+\beta_1L_1+\cdots+\beta_lL_l), 
\end{eqnarray*} 
for $(m,\beta)=(m,\beta_1,\cdots,\beta_l)\in \IN\times \IN^l$ such that $V_0 = \IC$ and $V_{m,\beta}\cdot V_{m',\beta'} \seq V_{m+m',\beta+\beta'}$. 
For any $\beta\in \IQ^l_{\ge0}$, we denote by $V_{(1,\beta)}$ the ($\IN$-)graded linear series $\{V_{m(1,\beta)}=V_{m,m\beta}\}_m$. 
\end{defi}
Basic notions for $R_\bu$ introduced above are similarly defined for $V_\bu$. For example, filtrations, (compatible) basis type divisors, $S$-invariants, and $\delta$-invariants. See \cite{AZ22,MW23} for details. 

\begin{defi}\rm
Let $E$ be a prime divisor over $X$. The {\it $E$-refinement} $W^E_\bu$ of $V_\bu$ is defined by 
$$W^E_{m,\beta,j} := \gr_E^j V_{m, \beta} = \CF_E^j V_{m,\beta}/ \CF_E^{j+1} V_{m,\beta}, \quad j\in \IN. $$
By \cite[Example 2.6]{AZ22}, if $E$ is of plt-type over $X$ or $E$ is Cartier on some birational model of $X$, then $W^E_\bu$ is a $\IN\times \IN^{l+1}$-graded linear series on $E$. 
\end{defi}
As a consequence, we have the following one-to-one correspondence 
\begin{eqnarray}
\label{Eqnarray: basis type divisors 1-1 correspondence}
\{m\text{-basis type divisors of }V_\bu\text{ compatible with }E\} 
\longleftrightarrow
\{m\text{-basis type divisors of }W^E_\bu\}. 
\end{eqnarray} 

We will deal with refinements by toric divisors in the following sections. 
Assume that $(X,\D)$ admits a $\IG_m$-action, $L,L_1,\cdots,L_l$ are $\IG_m$-linearized, and $V_{m,\beta}\seq H^0(X_l, mL+\beta_1L_1+\cdots+\beta_lL_l)$ is $\IG_m$-invariant. All these $\IG_m$-actions are assumed to be effective. Then we have weight decomposition
\begin{eqnarray*}
V_{m,\beta} = \bigoplus_{\alpha \in M(\IG_m)} V_{m,\beta,\alpha}. 
\end{eqnarray*}
Let $\rho\in N(\IG_m)$ be a primitive generator. Assume that there exists a prime divisor $E$ over $X$ and $c_E\in\IQ_{>0}, a_E \in \IQ$ such that $\ord_E(s) = c_E \cdot \la \alpha, \rho \ra + m a_{E}$ for sufficiently divisible $m$ and any $\alpha \in M(\IG_m), s\in V_{m,\beta,\alpha}$ (in particular $ma_E\in \IZ$). 
Then by a similar argument of Lemma \ref{Lemma: plt Fano type model}, we see that $c_E=1$. 
Hence the $E$-refinement $W^E_\bu$ of $V_\bu$ satisfies 
\begin{eqnarray}
\label{Eqnarray: W_(m,j) = V_(m,alpha)}
W^E_{m,\beta, j} \cong V_{m,\beta,\alpha}
\end{eqnarray}
for $j = \la\alpha,\rho\ra + m a_E$. In particular, we have the following one-to-one correspondence 
\begin{eqnarray}
\label{Eqnarray: basis type divisors 1-1 correspondence G_m}
\{\IG_m\text{-invariant }m\text{-basis type divisors of }V_\bu\} 
\longleftrightarrow
\{m\text{-basis type divisors of }W^E_\bu\}. 
\end{eqnarray} 

\begin{rmk}\rm
\label{Remark: virtual}
As explained by Lemma \ref{Lemma: plt Fano type model}, the $\IG_m$-action on $E$, hence on $W^E_\bu$, is trivial. However, the decomposition $W^E_{m,\beta,\bu} = \oplus_{j\in \IZ} W^E_{m,\beta,j}$ reveals a $\IG_m$-action, and we simply say that $W^E_\bu$ admits a {\it virtual} $\IG_m$-action. Under this assumption, the isomorphism (\ref{Eqnarray: W_(m,j) = V_(m,alpha)}) induces a $\IG_m$-equivariant isomorphism of $W^E_\bu$ and $V_\bu$.  
\end{rmk}


\section{The sharpness of Abban-Zhuang estimate}
\label{Section: The sharpness of Abban-Zhuang estimate}

In this section, we prove a sharpness result of Abban-Zhuang estimate \cite[Theorem 3.1]{AZ22}, see also \cite[Lemma 5.1]{MW23} for a $\IT$-equivariant version. We first recall the Abban-Zhuang estimate. 

Let $f:(X, \D)\to U$ be a $\IT$-equivariant projective morphism, where $U$ is a quasi-projective variety and $(X,\D)$ is a klt pair both admitting $\IT$-actions. 
We fix a $\IT$-invariant subvariety $Z\seq X$. Let $F$ be a $\IT$-invariant plt-type divisor over $X$ with $C_X(F)\supseteq Z$. We denote by $\pi: Y \to X$ the associated plt-type blowup, and by $\D_F = \Diff_{F}(\D_Y)$, $\D_Y=\pi_*^{-1}\D$. Then for any $\IT$-invariant multi-graded linear series $V_\bu$ on $X$ with $F$-refinement $W_{\bu}$, we have 
\begin{thm}[Abban-Zhuang]
\label{Theorem: Abban-Zhuang estimate}
We have the following estimate: 
\begin{eqnarray}
\label{Inequality: Abban-Zhuang estimate}
\delta_{Z,\IT}(X,\D;V_\bu) \ge 
\min\Big\{\frac{A_{X,\D}(F)}{S(V_\bu;F)},  
\mathop{\inf}_{Z'\mapsto Z} \delta_{Z',\IT}(F,\D_F;W_\bu)
\Big\}, 
\end{eqnarray}
where the infimum runs over all $\IT$-invariant closed subvarieties $Z'\seq F$ 
(with $\dim Z' = \dim Z$) 
mapping to $Z\seq X$. 
\end{thm}

The main result of this section is the following sharpness result of Abban-Zhuang estimate. 

\begin{thm}
\label{Theorem: Sharpness of AZ}
If $\ord_F$ is a minimizer of $\delta_{Z,\IT}(X,\D;V_\bu)$, then we have 
\begin{eqnarray*}
\frac{A_{X,\D}(F)}{S(V_\bu;F)} =
\delta_{Z,\IT}(X,\D;V_\bu) \le 
\mathop{\inf}_{Z'\mapsto Z} \delta_{Z',\IT}(F,\D_F;W_\bu). 
\end{eqnarray*}
Otherwise, we have strict inequalities 
\begin{eqnarray*}
\frac{A_{X,\D}(F)}{S(V_\bu;F)} >
\delta_{Z,\IT}(X,\D;V_\bu) > 
\mathop{\inf}_{Z'\mapsto Z} \delta_{Z',\IT}(F,\D_F;W_\bu). 
\end{eqnarray*}
\end{thm}

The theorem will help us to get some prior estimate of $\mathop{\inf}_{Z'\mapsto Z} \delta_{Z',\IT}(F,\D_F;W_\bu)$. 

\begin{defi}\rm
We define the {\it relative delta invariant} of $V_\bu$ with respect to $F$ as 
\begin{eqnarray}
\label{Eqnarray: relative delta invariant}
\delta_{Z,\IT}(X,\D,F;V_\bu):=
\mathop{\inf}_{v\in \Val^{\IT,\circ}_{X,Z}} f(v), 
\end{eqnarray}
where $f:\Val^{\IT,\circ}_{X,Z} \to (0, +\infty]$ is a function defined by 
\begin{eqnarray}
\label{Eqnarray: relative delta invariant f}
f(v) := \frac{A_{X,\D}(v)-A_{X,\D}(F)v(F)}{S(V_\bu;v)-S(V_\bu;F)v(F)}, 
\end{eqnarray}
when $v\ne c\cdot\ord_F$ for any $c\in \IR_{>0}$. And we define $f(\ord_F):=\frac{A_{X,\D}(F)}{S(V_\bu; F)}$. 
\end{defi}

\begin{rmk}\rm 
\label{Remark. both the numerator and the denominator are non-negative}
We remark that both the numerator and the denominator in (\ref{Eqnarray: relative delta invariant f}) are non-negative. 
Recall that $\pi:Y\to X$ is a plt type blowup with exceptional divisor $F$. We have 
\begin{eqnarray}
\label{Eqnarray: plt-type blowup, canonical bundle formula}
K_Y+\D_Y+(1-A_{X,\D}(F))F=\pi^*(K_X+\D), 
\end{eqnarray}
where $\D_Y=\pi_*^{-1}\D$. Hence 
$A_{X,\D}(v)-A_{X,\D}(F)v(F) 
=
A_{Y,\D_Y+F}(v) \ge 0$, and the equality holds if and only if $v=c\cdot \ord_F$ for some $c>0$, since $(Y, \D_Y+F)$ is plt. 

On the other hand, we may choose a $\IT$-invariant $m$-basis type divisor $D$ of $V_\bu$ compatible with both $F$ and $v$. Then $\pi^*D=S_m(V_\bu;F)F+\Gamma$ where $\Gamma$ is effective and does not contain $F$ as a component. Then $v(D) = S_m(V_\bu;v)$. Hence
$$S_m(V_\bu;v)-S_m(V_\bu;F)v(F)=v(D)-S_m(V_\bu;F)v(F) = v(\Gamma) \ge 0. $$
Taking $m\to \infty$ we see that $S(V_\bu;v)-S(V_\bu;F)v(F) \ge 0$. 
\end{rmk}

\begin{thm}
\label{Theorem: relative delta = min}
\begin{eqnarray*}
\delta_{Z,\IT}(X,\D,F;V_\bu) =
\min\Big\{\frac{A_{X,\D}(F)}{S(V_\bu;F)}, 
\mathop{\inf}_{Z'\mapsto Z} \delta_{Z',\IT}(F,\D_F;W_\bu)
\Big\}. 
\end{eqnarray*}
\end{thm}

The proof of this theorem follows from a refined version of Abban-Zhuang's original argument using basis type divisors. For any $\IT$-invariant boundary $V$ on $X$ (see for example \cite[Section 5.1]{MW23}), we define the relative delta invariant of $V$ with respect to $F$ by 
\begin{eqnarray}
\label{Eqnarray: relative delta invariant for boundary}
\delta_{Z,\IT}(X,\D,F;V):=
\mathop{\inf}_{v\in \Val^{\IT,\circ}_{X,Z}} \frac{A_{X,\D}(v)-A_{X,\D}(F)v(F)}{S(V;v)-S(V;F)v(F)} =
\mathop{\inf}_{v\in \Val^{\IT,\circ}_{X,Z}} \frac{A_{Y,\D_Y+F}(v)}{v(\Gamma)}, 
\end{eqnarray}
where $\Gamma=\pi^*D - S(V;F)F$, and $D$ is a $\IT$-invariant basis type divisor of $V$ compatible with both $F$ and $v$. 
Then the $m$-th relative delta invariant of $V_\bu$ with respect to $F$ is defined by choosing $V=V_m$, 
\begin{eqnarray}
\label{Eqnarray: m-th relative delta invariant}
\delta_{Z,\IT,m}(X,\D,F;V_\bu) :=
\delta_{Z,\IT}(X,\D,F;V_m). 
\end{eqnarray}
By \cite[Corollary 2.10]{BJ20} or \cite[Lemma 3.2]{MW23}, it's not difficult to show that (\ref{Eqnarray: m-th relative delta invariant}) converges to (\ref{Eqnarray: relative delta invariant}) as $m\to \infty$. Hence it remains to prove the following non-graded version of Theorem \ref{Theorem: relative delta = min}. 

\begin{lem}
\label{Lemma: relative delta = min}
For any $\IT$-invariant boundary $V$ on $X$ with $F$-refinement $W$, we have 
\begin{eqnarray*}
\delta_{Z,\IT}(X,\D,F;V) =
\min\Big\{\frac{A_{X,\D}(F)}{S(V;F)}, 
\mathop{\inf}_{Z'\mapsto Z} \delta_{Z',\IT}(F,\D_F;W)
\Big\}. 
\end{eqnarray*}
\end{lem}

\begin{proof}
For simplicity, we denote by 
\begin{eqnarray*}
\mu = \delta_{Z,\IT}(X,\D,F;V), \quad
\lam = \min\Big\{\frac{A_{X,\D}(F)}{S(V;F)}, 
\mathop{\inf}_{Z'\mapsto Z}\delta_{Z',\IT}(F,\D_F;W)
\Big\}. 
\end{eqnarray*}
Fix a constant $0\le\eta\le A_{X,\D}(F)/S(V;F)$. In particular, we may choose $\eta=\lam$ or $\mu$. 

Let $D$ be a $\IT$-invariant basis type divisor of $V$ compatible with $F$. Then $\pi^*D=S(V;F)F + \Gamma^D$, where $\Gamma^D$ is an effective $\IR$-divisor not containing $F$ as a component. Note that $\Gamma^D|_F$ is a $\IT$-invariant basis type divisor of $W$, and the map 
\begin{eqnarray}
\label{Eqnarray: D to Gamma|_F, one-to-one correspondence}
\Omega_{V,F} \longrightarrow \Omega_W, \,\,\,
D \mapsto \Gamma^D|_F
\end{eqnarray}
is a one-to-one correspondence, where 
\begin{eqnarray*}
\Omega_{V,F} &:=& \{\IT\text{-invariant basis type divisors of } V \text{ compatible with } F\}, \\
\Omega_{W} &:=& \{\IT\text{-invariant basis type divisors of } W\}. 
\end{eqnarray*}
On the other hand, we have 
\begin{eqnarray*}
K_Y+\D_Y+ a_\eta F +\eta \Gamma^D=\pi^*(K_X+\D+\eta D), 
\end{eqnarray*}
where $a_\eta = 1-A_{X,\D}(F)+\eta S(V;F)\le 1$. Hence by inversion of adjunction, we have
\begin{eqnarray}
\label{Eqnarray. lc at eta_Z' equivalence}
(Y,\D_Y+F+\eta\Gamma^D)\text{ is lc at }\eta_{Z'} 
\,\,\, \Longleftrightarrow \,\,\, 
(F, \D_F+\eta\Gamma^D|_F)\text{ is lc at }\eta_{Z'}, 
\end{eqnarray}
where $Z'\seq F$ is a $\IT$-invariant subvariety mapping to $Z$ and $\eta_{Z'}$ is the generic point of $Z'$.

\begin{lem}
\label{Lemma: adjunction equivalences}
We have the following equivalences. 

{\rm (1)} The pair $(Y,\D_Y+F+\eta\Gamma^D)$ is lc at $\eta_{Z'}$ for any $\IT$-invariant subvariety $Z'\seq F$ mapping to $Z$ and for any $D\in \Omega_{V,F}$ if and only if $\eta \le \mu$. 

{\rm (2)} The pair $(F, \D_F+\eta D')$ is lc at $\eta_{Z'}$ for any $\IT$-invariant subvariety $Z'\seq F$ mapping to $Z$ and for any $D'\in \Omega_W$ if and only if $\eta \le \lam$. 
\end{lem}

\begin{proof} 
We first prove the equivalence (1). 
The former condition of (1) is equivalent to 
\begin{eqnarray*}
0 &\le& A_{Y, \D_Y + F + \eta \Gamma^D}(v) \\
&=& A_{Y, \D_Y + a_\eta F +(1-a_\eta)F + \eta \Gamma^D}(v) \\
&=& A_{X,\D+\eta D}(v) - (1-a_\eta) v(F) \\ 
&=& A_{X,\D}(v) - \eta v(D) -(A_{X,\D}(F)-\eta S(V;F)) v(F) \\
&=& \big(A_{X,\D}(v) - A_{X,\D}(F)v(F)\big) -  \eta\cdot \big(v(D) - S(V;F)v(F)\big)
\end{eqnarray*}
for any $\IT$-invariant subvariety $Z'\seq F$ mapping to $Z$, any $v \in \Val^{\IT,\circ}_{Y,Z'}$ and any $D\in \Omega_{V,F}$. Recall that 
$$S(V;v) = \sup_{D\in \Omega_{V,F}} v(D).$$
Hence the former condition of (1) is equivalent to 
\begin{eqnarray*}
0 \,\,\,\le\,\,\, \big(A_{X,\D}(v) - A_{X,\D}(F)v(F)\big) -  \eta \cdot \big(S(V;v) - S(V;F)v(F)\big) 
\end{eqnarray*}
for any $\IT$-invariant subvariety $Z'\seq F$ mapping to $Z$ and any $v \in \Val^{\IT,\circ}_{Y,Z'}$, equivalently, for any $v\in \Val^{\IT,\circ}_{X,Z}$ since 
$$\mathop{\cup}_{Z'\mapsto Z}\Val^{\IT,\circ}_{Y,Z'} = \Val^{\IT,\circ}_{X,Z}. $$
We conclude that the former condition of (1) is equivalent to 
\begin{eqnarray*}
\eta \le \mathop{\inf}_{v\in \Val^{\IT,\circ}_{X,Z}} \frac{A_{X,\D}(v) - A_{X,\D}(F)v(F)}{S(V;v) - S(V;F)v(F)}  = \mu. 
\end{eqnarray*}

\begin{rmk} \rm
Here we used the fact that $S(V;v) - S(V;F)v(F) \ge 0$. See Remark \ref{Remark. both the numerator and the denominator are non-negative}. 
\end{rmk}

Next, we prove the equivalence (2). The former condition of (2) is equivalent to 
\begin{eqnarray*}
0 &\le& A_{F, \D_F+\eta D'}(w) \\
&=& A_{F, \D_F}(w) - \eta \cdot w(D')
\end{eqnarray*}
for any $\IT$-invariant subvariety $Z'\seq F$ mapping to $Z$, any $w\in \Val^{\IT,\circ}_{F,Z'}$ and any $D'\in \Omega_{W}$. Recall that 
$$S(W;w) = \sup_{D'\in \Omega_{W}} w(D').$$
Hence the former condition of (2) is equivalent to 
\begin{eqnarray*}
0 \,\,\,\le\,\,\, 
 A_{F, \D_F}(w) - \eta \cdot S(W;w) 
\end{eqnarray*}
for any $\IT$-invariant subvariety $Z'\seq F$ mapping to $Z$ and any $w\in \Val^{\IT,\circ}_{F,Z'}$. 
We conclude that the former condition of (1) is equivalent to 
\begin{eqnarray*}
\eta \le 
\mathop{\inf}_{Z'\mapsto Z} 
\mathop{\inf}_{w\in \Val^{\IT,\circ}_{F,Z'}}  
\frac{A_{F,\D_F}(w)}{S(W;w)} 
= 
\mathop{\inf}_{Z'\mapsto Z} 
\delta_{Z',\IT}(F,\D_F;W)
= \lam. 
\end{eqnarray*}
The proof of Lemma \ref{Lemma: adjunction equivalences} is finished. 
\end{proof}

If we choose $\eta = \lam$, then $(F, \D_F+\lam \Gamma^D|_F)$ is lc at $\eta_{Z'}$ by Lemma \ref{Lemma: adjunction equivalences} (2), hence $(Y,\D_Y+F+\lam\Gamma^D)$ is lc at $\eta_{Z'}$ by (\ref{Eqnarray. lc at eta_Z' equivalence}) for any $\IT$-invariant subvariety $Z'\seq F$ mapping to $Z$ and for any $D\in \Omega_{V,F}$. We see that $\lam \le \mu$ by Lemma \ref{Lemma: adjunction equivalences} (1). Reversing the argument, we also have $\mu\le \lam$. Hence $\lam = \mu$. The proof is completed. 
\end{proof}

\begin{cor}
\label{Corollary: Sharpness of AZ}
If $\ord_F$ minimizes $\delta_{Z,\IT}(X,\D;V)$, then it minimizes $\delta_{Z,\IT}(X,\D,F;V)$. In particular 
\begin{eqnarray*}
\frac{A_{X,\D}(F)}{S(V;F)} =
\delta_{Z,\IT}(X,\D;V) \le 
\mathop{\inf}_{Z'\mapsto Z}\delta_{Z',\IT}(F,\D_F;W). 
\end{eqnarray*}
Otherwise, we have strict inequalities 
\begin{eqnarray*}
\frac{A_{X,\D}(F)}{S(V;F)} >
\delta_{Z,\IT}(X,\D;V) >
\mathop{\inf}_{Z'\mapsto Z}\delta_{Z',\IT}(F,\D_F;W). 
\end{eqnarray*}
\end{cor}
\begin{proof}
We first see that $\delta_{Z,\IT}(X,\D;V) \ge \delta_{Z,\IT}(X,\D,F;V)$ by $\IT$-equivariant Abban-Zhuang estimate \cite[Lemma 5.1]{MW23} and Lemma \ref{Lemma: relative delta = min}. 
For any $v\in \Val^{\IT,\circ}_{X,Z}$, we denote by 
\begin{eqnarray*}
a=A_{X,\D}(v), \,\,\,
b=S(V; v), \,\,\,
a_0=A_{X,\D}(F)v(F), \,\,\,
b_0=S(V; F)v(F). \,\,\,
\end{eqnarray*}

If $\ord_F$ minimizes $\delta_{Z,\IT}(X,\D;V)$ but $v$ does not, we see that $v$ also does not minimize $\delta_{Z,\IT}(X,F,\D;V)$ by the following elementary inequality
\begin{eqnarray}
\label{Eqnarray: elementary inequality}
\frac{a_0}{b_0} < \frac{a}{b} 
\,\,\,\Longleftrightarrow\,\,\,
\frac{a_0}{b_0} <\frac{a}{b} < \frac{a-a_0}{b-b_0},  
\end{eqnarray}
since $b_0, b-b_0>0$. The first assertion follows. 

Otherwise, let $v$ be a minimizer of $\delta_{Z,\IT}(X,\D;V)$, then it follows by
\begin{eqnarray*}
\frac{a_0}{b_0} > \frac{a}{b} 
\,\,\,\Longleftrightarrow\,\,\,
\frac{a_0}{b_0} >\frac{a}{b} > \frac{a-a_0}{b-b_0}. 
\end{eqnarray*}
\end{proof}

\begin{proof}[Proof of Theorem \ref{Theorem: relative delta = min}]
Applying Lemma \ref{Lemma: relative delta = min} to $V= V_m$. Then it follows from the convergence of delta invariant $\lim_{m\to \infty}\delta_m = \delta$, see \cite[Corollary 2.10]{BJ20} or \cite[Lemma 3.2]{MW23}. 
\end{proof}

\begin{proof}[Proof of Theorem \ref{Theorem: Sharpness of AZ}]
It follows from Theorem \ref{Theorem: relative delta = min} and the same argument of Corollary \ref{Corollary: Sharpness of AZ}. 
\end{proof}

\begin{rmk}\rm
It's not difficult to generalize Theorem \ref{Theorem: relative delta = min} and \ref{Theorem: Sharpness of AZ} to the $g$-weighted setting. We need only to replace $V=V_m$ by the $g$-weighted boundary $V=V^g_m$ (see \cite[Definition 5.2]{MW23}) in the proof of Theorem \ref{Theorem: relative delta = min}. 
\end{rmk}

\section{A Valuative Criterion of (Weighted) K-polystability}
\label{Section: A Valuative Criterion of (Weighted) K-polystability}


Let $(X,\D)$ be a log Fano pair of dimension $n$, and $\IT\seq \Aut(X,\D)$ be a subtorus of rank $r$. Then the $\IT$-action lifts to the canonical divisor $K_X+\D$. Hence the anti-canonical ring $R=R(X,\D)=\oplus_{m\ge 0} R_m$ admits a weight decompsition $R_m= \oplus_{\alpha\in M}R_{m,\alpha}$, where $M=\Hom(\IT, \IG_m)$ is the weight lattice. Let $N=\Hom(\IG_m, \IT)=M^\vee$ be the lattice of one parameter subgroups. Assume that $N$ is generated by $\xi_1, \cdots, \xi_r$, and we denote by $\la\xi_i\ra$ the subtorus of $\IT$ generated by $\xi_i$. 

\subsection{Construction of $(X_r,\D_r,W^{X_r}_\bu)$}
\begin{const} \rm
\label{Construction: toric divisor sequence}
We construct a sequence of klt Fano type pairs $(X_i, \D_i)$ of dimension $(n-i)$ by induction on $0\le i\le r$ such that 

(a) it admits an effective $\IT_i$-action, where $\IT_i = \la\xi_{i+1}\ra \times \cdots \times \la\xi_r\ra \seq \IT$; 

(b) there exists a $\IT_i$-invariant $\IQ$-divisor $D_i\ge \D_i$ on $X_i$ such that $(X_i,D_i)$ is a log Fano pair; 

(c) there is an isomorphism $\Val_{X_{i}}\times N(\la\xi_i\ra)_\IR\to \Val^{\la\xi_i\ra}_{X_{i-1}}$ such that 
$$A_{X_{i-1},\D_{i-1}}(v)=A_{X_i,\D_i}(v), \quad \forall v\in \Val_{X_i}. $$

Let $(X_0,\D_0):= (X,\D)$ and $D_0=\D_0$. Assume that $(X_{i-1},\D_{i-1}, D_{i-1})$ is defined. We construct $(X_i, \D_i, D_i)$ by using the $\la\xi_i\ra$-action. By Lemma \ref{Lemma: plt Fano type model}, there exists a $\IT_{i-1}$-equivariant plt-type blowup $\pi_{i-1}: (Y_{i-1}, X_i)\to (X_{i-1}, D_{i-1})$ such that $X_i$ is the toric divisor over $X_{i-1}$ with respect to the $\la\xi_i\ra$-action. 
Hence the $\la\xi_i\ra$-action on $X_i$ is trivial and the $\IT_i$-action on $X_i$ is effective by induction hypothesis (the $\IT_{i-1} = \la\xi_i\ra \times \IT_i$-action on $X_{i-1}$ is effective). Let $\Gamma_{i-1}$ be an $\IT_{i-1}$-invariant effective $\IQ$-divisor on $Y_{i-1}$ such that $\Gamma_{i-1}+ X_i \ge \pi_{i-1,*}^{-1}D_{i-1}$,  $(Y_{i-1}, \Gamma_{i-1}+X_i)$ is plt and $-(K_{Y_{i-1}}+\Gamma_{i-1}+X_i)$ is ample. We define
$$D_i := \Diff_{X_i}(\Gamma_{i-1}) \ge
 \Diff_{X_i}(\pi_{i-1,*}^{-1} \D_{i-1}) =: \D_i. $$
By adjunction, we have $-(K_{X_i}+D_i)=
-(K_{Y_{i-1}}+\Gamma_{i-1}+X_i)|_{X_i}$, and $(X_i, D_i)$ is klt. Hence it is a log Fano pair. 
Now we get a $(n-i)$-dimensional Fano type klt pair $(X_i,\D_i)$ with an effective $\IT_i$-action. The isomorphism in (c) and the equality of log discrepancies follow directly from Lemma \ref{Lemma: plt Fano type model}. 

Since $X_i$ is of plt type over $(X_{i-1}, \D_{i-1})$, we can  inductively define a $\IN\times \IN^i$-graded linear series $W^{X_i}_\bu$ {\it on} $X_i$ by letting $W^{X_0}_\bu = R_\bu = R(X,\D)$, and letting $W^{X_i}_\bu$ be the $X_i$-refinement of $W^{X_{i-1}}_\bu$. 
\end{const}

\begin{lem} 
\label{Lemma: A=A, S=S}
There exists an isomorphism $i: \Val_{X_r}\times N_\IR\to \Val^{\IT}_X$ such that 
$$A_{X,\D}(v)=A_{X_r,\D_r}(v), \quad
S(R_\bu; v) =S(W^{X_r}_\bu; v), $$
for any $v\in \Val_{X_r}$ (still denote $i(v,0)$ by $v$). 
\end{lem}
\begin{proof}
The isomorphism $i$ and the equality of log discrepancies are obtained by succeedingly using Lemma \ref{Lemma: plt Fano type model}. 
For the equality of $S$-invariants, it follows from the one-to-one correspondence (\ref{Eqnarray: basis type divisors 1-1 correspondence G_m}) of basis type divisors. Indeed, let $D$ be a $\IT$-invariant $m$-basis type divisor of $R_\bu$ compatible with $v$. Then using (\ref{Eqnarray: basis type divisors 1-1 correspondence G_m}) succeedingly, we get a $m$-basis type divisor $D_r$ of $W^{X_r}_\bu$ compatible with $v$. Hence 
$S_m(R_\bu; v) = v(D) = v(D_r) = S_m(W^{X_r}_\bu; v) $
for sufficiently divisible $m$. Taking $m\to \infty$ we get the required equality. 
\end{proof}

\subsection{Proof of the main theorem}
We are ready to prove the main theorem. 

\begin{thm} 
\label{Theorem: main thm}
Assume that $\IT\seq \Aut(X,\D)$ is a maximal subtorus and $\Fut|_N=0$. Then the log Fano pair $(X,\D)$ is K-polystable (K-semistable) if and only if 
$$\delta(X_r, \D_{X_r}; W^{X_r}_\bu)> (\ge) 1. $$
\end{thm}

\begin{rmk}\rm
The key ingredient of the theorem is that, to test K-stability, we need only to compute delta around only one irreducible $\IT$-fixed subvariety $C_X(X_r)$ of $X$. 
\end{rmk}

\begin{proof}
First note that the assumption $\Fut|_N=0$ ensures that 
\begin{eqnarray*}
\frac{A_{X_{i-1}, \D_{i-1}}(X_i)}{S(W^{X_{i-1}}_\bu; X_i)} = 
\frac{A_{X,\D}(X_i)}{S(R_\bu; X_i)} =1, 
\end{eqnarray*} 
for any $1\le i\le r$, where $X_i$ is viewed as a divisorial valuation over $X$ by Lemma \ref{Lemma: A=A, S=S}. 

If $(X,\D)$ is K-semistable, then $\delta(X_r, \D_r; W^{X_r}_\bu)\ge 1$ by the sharpness of Abban-Zhuang estimate Theorem \ref{Theorem: Sharpness of AZ}. Conversely, assume that $(X,\D)$ is K-unstable. Let $w\in\Val_{X}^\IT=\Val_{X_r}\times \IN_\IR$ be a destablizing valuation. Then $w= v_\xi$ for some $v\in \Val_{X_r}$ and $\xi\in N_\IR$. Hence
\begin{eqnarray*}
1 >
\frac{A_{X,\D}(w)}{S(R_\bu; w)}
=
\frac{A_{X,\D}(v) + \theta_\xi(v)}{S(R_\bu; v)+ \theta_\xi(v)} 
=
\frac{A_{X_r,\D_r}(v) + \theta_\xi(v)}{S(W^{X_r}_\bu; v)+ \theta_\xi(v)}, 
\end{eqnarray*} 
where the first equality follows from $\Fut|_N = 0$, (\ref{Eqnarray: A(v_xi) = A(v)+theta_xi(v)}) and (\ref{Eqnarray: S(v_xi) = S(v)+theta_xi(v)}); the second equality follows from Lemma \ref{Lemma: A=A, S=S}. 
By the elementary inequality (\ref{Eqnarray: elementary inequality}), if $\theta_\xi(v)\ge 0$, then we have 
\begin{eqnarray*}
1 >
\frac{A_{X_r,\D_r}(v) + \theta_\xi(v)}{S(W^{X_r}_\bu; v)+ \theta_\xi(v)} 
\ge 
\frac{A_{X_r,\D_r}(v)}{S(W^{X_r}_\bu; v)}; 
\end{eqnarray*} 
else $\theta_\xi(v)< 0$, then 
\begin{eqnarray*}
1 > 
\frac{A_{X_r,\D_r}(v)}{S(W^{X_r}_\bu; v)}
>
\frac{A_{X_r,\D_r}(v) + \theta_\xi(v)}{S(W^{X_r}_\bu; v)+ \theta_\xi(v)}. 
\end{eqnarray*} 
We get a contradiction since $\delta(X_r, \D_r; W^{X_r}_\bu)\ge 1$. 

Next, we prove the second assertion. If $(X,\D)$ is K-polystable, then it is K-semistable. Hence $\delta(X_r, \D_r; W^{X_r}_\bu)\ge 1$ by the previous paragraph. Assume that $\delta(X_r, \D_r; W^{X_r}_\bu)= 1$, then using the same argument of \cite{BJ20}, there exists a valuation $v\in \Val_{X_r}$ such that 
\begin{eqnarray*}
1 = \frac{A_{X_r,\D_r}(v)}{S(W^{X_r}_\bu; v)} = 
\frac{A_{X,\D}(v)}{S(R_\bu; v)} =
\frac{A_{X,\D}(v) + \theta_\xi(v)}{S(R_\bu; v)+ \theta_\xi(v)} =
\frac{A_{X,\D}(v_\xi)}{S(R_\bu; v_\xi)}, 
\end{eqnarray*} 
for any $\xi\in N_\IR$. Hence $\delta^\red_\IT(X,\D)\le1$. By \cite{LXZ22}, the log Fano pair $(X,\D)$ is not K-polystable. We get a contradiction. 

Conversely, if $(X,\D)$ is K-semistable but not K-polystable, then by \cite{XZ19} there exists $w = v_\xi \in\Val_{X_r}\times \IN_\IR = \Val_{X}^\IT$ with $\xi\in N_\IR$ and non-trivial $v\in \Val_{X_r}$ such that 
\begin{eqnarray*}
1 =
\frac{A_{X,\D}(w)}{S(R_\bu; w)}
=
\frac{A_{X,\D}(v) + \theta_\xi(v)}{S(R_\bu; v)+ \theta_\xi(v)} 
=
\frac{A_{X_r,\D_r}(v) + \theta_\xi(v)}{S(W^{X_r}_\bu; v)+ \theta_\xi(v)} 
=
\frac{A_{X_r,\D_r}(v)}{S(W^{X_r}_\bu; v)}. 
\end{eqnarray*} 
Hence $\delta(X_r, \D_r; W^{X_r}_\bu)\le 1$. 
\end{proof}

\begin{rmk}\rm 
In the last paragraph of the proof, we indeed showed that $\delta(X_r, \D_r; W^{X_r}_\bu)> 1$ implies $(X,\D)$ being reduced uniformly K-stable. To prove K-polystability directly, we may choose $w$ to be the valuation induced by some non-product type $\IT$-invariant special TC with vanishing $\Fut$. 
\end{rmk}

\begin{ex}[Plane conics]\rm
\label{Example: plane conics}
As the first example, we consider the log Fano pair $(\IP^2, cQ)$ where $Q\seq \IP^2$ is a smooth conic curve. It's well-known by \cite{LS14} that this pair is K-polystable (K-semistable) if and only if $0\le c < (\le) \frac{3}{4}$. We state another proof based on Theorem \ref{Theorem: main thm}. 

We may assume that $Q=\{xz-y^2=0\}\seq \IP^2_{x,y,z}$ and consider the $\IG_m$-action $t\cdot [x,y,z] = [x,ty,t^2z]$, which generates a maximal torus of $\Aut(\IP^2,cQ)\cong \SL_2$. Let $Y\to \IP^2$ be the $(2,1)$-blowup at the point $[1,0,0]$ with exceptional divisor $E$ such that $\ord_E(y)=1$ and $\ord_E(z)=2$. Then $E$ is a toric divisor of the $\IG_m$-action. We denote by $\tQ$ and $\tl$ the strict transform of $Q$ and $l=\{z=0\}$ respectively. Let $p_0\in E$ be the unique singular point of $Y$, $p_1=\tl\cap E$ and $p_2 = \tQ\cap E$. Then $p_0,p_1,p_2$ are mutually distinct, $\Diff_E(0) = \frac{1}{2}p_0$ and $\Diff_E(cQ) = \frac{1}{2}p_0 + c\, p_2$. 

Now let's take refinement of $-(K_{\IP^2}+cQ) = \CO(3-2c)$ by $E$. It suffices to refine $\CO(1)$. 
First note that $\CO(1)-tE$ is ample for $0<t<1$. For $1\le t\le 2$, we have the following Zariski decomposition 
\begin{eqnarray*}
\CO(1)-tE = (2-t)(\CO(1)-E) + (t-1) \tl. 
\end{eqnarray*} 
Hence $S(\CO(1);E) = \frac{1}{\vol(\CO(1))} \int_0^2 \vol(\CO(1)-tE)dt = 1$, and $$\FL(E) = A_{\IP^2,cQ}(E) - S(\CO(3-2c);E) = 0. $$ 
In other word, $\Fut=0$ on $N(\IG_m)$. 

The refinement $W_\bu$ of $\CO(1)$ by $E\cong \IP^1$ follows by 
\begin{eqnarray*}
W^{E}_{(1,t)} = 
\left\{ \begin{array}{ll}
H^0\Big(\IP^1, \CO(\frac{t}{2})\Big) 
& 0\le t < 1,\\
H^0\Big(\IP^1, \CO(1-\frac{t}{2})\Big) + (t-1)\cdot p_1
& 1\le t\le 2. 
\end{array} \right. 
\end{eqnarray*}
Hence $S(W_\bu; p) = \frac{1}{6}$ for any $p\ne p_1$, and $S(W_\bu; p_1) = \frac{1}{3}$. They should be replaced by the multiplication with $(3-2c)$ if we replace $W_\bu$ by $W^E_\bu$, which is the refinement of $-K_{\IP^2}-cQ$ by $E$. Hence 
$$\delta(E, \frac{1}{2}p_0 + c\, p_2; W^E_\bu) = \min\{\frac{3}{3-2c},\frac{6-6c}{3-2c}\},$$ which is $ > (\ge) 1$ if and only if $0< c<(\le)\frac{3}{4}$. We conclude by Theorem \ref{Theorem: main thm}. 

\end{ex}

\begin{defi}\rm
We define the following multi-graded version of the Fujita-Li invariant 
$$\FL(W^{X_r}_\bu;v) := A_{X_r,\D_r}(v) - S(W^{X_r}_\bu;v), $$
for any valuation $v$ over $X_r$. 
\end{defi}

As a consequence, we have the following valuative criterion of K-polystability. 

\begin{cor}
Let $(X,\D)$ be a log Fano pair with a maximal torus $\IT=\IG_m^r$-action. Then it is K-polystable if and only if 
$\FL(W^{X_r}_\bu;v) > 0$ for any valuation $v$ over $X_r$.  
\end{cor}

It is worth noting the generalization of Theorem \ref{Theorem: main thm} to the $g$-weighted setting (see \cite{MW23}). This case has the advantage that the modified Futaki invariant $\Fut_g$ is automatically vanishing on $N$. 

\begin{thm} 
Let $g:\BP\to \IR_{>0}$ be a weight function (see (\ref{Eqnarray: Fut_g = 0})) and assume that $\IT\seq \Aut(X,\D)$ is a maximal. Then the log Fano pair $(X,\D)$ is $g$-weighted K-polystable (K-semistable) if and only if 
$$\delta^g(X_r, \D_{X_r}; W^{X_r}_\bu)> (\ge) 1. $$
\end{thm}

\begin{proof}
The proof is the same as the previous one, only replacing the $S$-invariant by the $g$-weighted version, that is, $S^g$-invariant, see \cite{MW23}. 
\end{proof}

\section{Applications: existence of $g$-solitons}
\label{Section: Applications: almost complete linear series}


In this section, we give a more detailed study of the triple $(X_r,\D_r,W^{X_r}_\bu)$ in Theorem \ref{Theorem: main thm}. 
The base polytope of $W^{X_r}_\bu$ (see \cite[Section 2.7]{MW23}) is a shifting of the moment polytope $\BP$ of $R_\bu$ with respect to the $\IT$-action, that is, 
$R_{m(1,\alpha_1,\cdots, \alpha_r)} \cong W^{X_r}_{m(1,\alpha+a_1,\cdots,\alpha_r+a_r)}, $
where $a_i=A_{X,\D}(\ord_{X_i})$. We define the {\it normalized} linear series $W_\bu$ by 
\begin{eqnarray}
W_{m(1,\alpha_1,\cdots, \alpha_r)} := W^{X_r}_{m(1,\alpha_1+a_1,\cdots,\alpha_r+a_r)},\quad \alpha\in \BP_\IQ. 
\end{eqnarray}
Hence the base polytope of $W_\bu$ is the same as the moment polytope of $R_\bu$ with respect to the $\IT$-action.

Fix a continuous function $g:\BP\to \IR_{>0}$ with the property 
\begin{eqnarray}
\label{Eqnarray: Fut_g = 0}
\int_\BP \alpha_i\cdot g(\alpha) \vol(W_{(1,\alpha)}) d\alpha = 0, \quad 1\le i \le r. 
\end{eqnarray}
In other word, the modified Futaki invariant $\Fut_g|_N = 0$. 
The function $g$ is called a {\it weight function}. 
We will work in the $g$-weighted setting in this section. The $g$-weighted volume of $W_\bu$ is 
\begin{eqnarray}
\label{Eqnarray: v^g}
\BV^g=\int_\BP g(\alpha) \vol(W_{(1,\alpha)}) d\alpha. 
\end{eqnarray}
We may define the $g$-weighted measure $\DH^g_\BP$ (see \cite[Section 3.3]{MW23}) on $\BP$ by  
\begin{eqnarray*}
\DH^g_\BP(d \alpha) :=\frac{1}{\BV^g} \cdot g(\alpha) \vol(W_{(1,\alpha)}) d\alpha. 
\end{eqnarray*}
Then (\ref{Eqnarray: Fut_g = 0}) and (\ref{Eqnarray: v^g}) can be reformulated as
\begin{eqnarray}
\label{Eqnarray: Int DH^g_P =1, Int alpha_i DH^g_P =0}
\int_\BP \alpha_i\cdot \DH^g_\BP(d\alpha) = 0, \quad
\int_\BP \DH^g_\BP(d\alpha) = 1. 
\end{eqnarray}

\subsection{Almost complete condition}
Let $L$ be a big line bundle on $X_r$. Assume that there exist continuous functions $f, k_j$ on $\BP\seq \IR^r$ and prime divisors $F_j$ on $X_r$ such that we have the following decomposition of linear series
\begin{eqnarray}
\label{Eqnarray: almost complete assumption}
W_{(1,\alpha)} = R(f(\alpha) L) + F(\alpha) 
\end{eqnarray}
for any $\alpha \in \BP$, where $R(f(\alpha)L)$ is the graded linear series generated by $f(\alpha)L$, and $F(\alpha)=\sum_j k_j(\alpha) F_j$ is the fixed part of the linear series $W_{(1,\alpha)}$. This is just the {\it almost complete} condition introduced by \cite{AZ22} and the equality in (\ref{Eqnarray: almost complete assumption}) could be weakened to asymptotical equivalence. We use the form (\ref{Eqnarray: almost complete assumption}) for simplicity of notions. 


\begin{lem}
Under the assumption (\ref{Eqnarray: almost complete assumption}), for any $v\in \Val_X$,  we have 
\begin{eqnarray}
S^g(W_\bu, v) = \lam\cdot S(L; v) + v(F)
\end{eqnarray}
where $\lam=\int_\BP f(\alpha)\cdot \DH^g_\BP(d\alpha), F = 
\int_\BP F(\alpha)\cdot \DH^g_\BP(d\alpha)$ and $$\DH^g_\BP(d\alpha) = \frac{1}{\BV^g}\vol(L) f(\alpha)^{n-r} g(\alpha)d\alpha. $$ 
\end{lem}

\begin{proof}
We denote by $\CF=\CF_v$ the filtration induced by $v$. Then
\begin{eqnarray*}
\CF^{(t)}W_{(1,\alpha)} = 
\left\{ \begin{array}{ll}
W_{(1,\alpha)} & t\le v(F(\alpha)), \\
\CF^{(t- v(F(\alpha)))} R(f(\alpha) L) & t> v(F(\alpha)). 
\end{array} \right. 
\end{eqnarray*}
Hence
\begin{eqnarray*}
S(W_{(1,\alpha)};v) 
&=& \frac{1}{\vol(W_{(1,\alpha)})} 
\big(\int_{v(F(\alpha))}^\infty + \int_0^{v(F(\alpha))} \big) \vol(\CF^{(t)}W_{(1,\alpha)}) dt \\
&=& f(\alpha)\cdot S(L; v)+ v(F(\alpha)). 
\end{eqnarray*}
Then taking integration for $\alpha \in \BP$ with measure $\DH^g_\BP$ we get
\begin{eqnarray*} 
S^g(W_\bu; v)
&=& \int_\BP S(W_{(1,\alpha)};v) \cdot \DH^g_\BP(d\alpha) \\
&=& \lam\cdot S(L;v) + v(F). 
\end{eqnarray*}
The $g$-weighted DH measure follows from 
$$\vol(W_{(1,\alpha)}) = \vol(f(\alpha)L) = f(\alpha)^{n-r} \vol(L). $$
\end{proof}

\begin{thm}
\label{Theorem: main thm of almost complete}
Under the assumption (\ref{Eqnarray: almost complete assumption}), we have $\lam L = -(K_{X_r}+\D_r+F)$. Moreover, $(X,\D)$ is $g$-weighted K-semistable (K-polystable) if and only if $(X_r, \D_r+F)$ is K-semistable (K-stable or K-polystable). 
\end{thm}

\begin{proof}
Recall that $W_\bu$ is defined by (we omit pull-back morphisms)
\begin{eqnarray*} 
W_{(1,\alpha_1,\cdots, \alpha_r)} = (\cdots(-(K_X+\D)-t_1X_1)|_{X_1} - \cdots -t_r X_r)|_{X_r}, 
\end{eqnarray*}
where $t_i=\alpha_i +A_{X,\D}(X_i)$. Taking integration on $\BP$ with measure $\DH^g_\BP$ we get by (\ref{Eqnarray: Int DH^g_P =1, Int alpha_i DH^g_P =0}) 
\begin{eqnarray*} 
\int_\BP W_{(1,\alpha)}\cdot \DH^g_\BP(d \alpha) 
&=& (\cdots(-(K_X+\D)-A_{X,\D}(X_1) X_1)|_{X_1} - \cdots - A_{X,\D}(X_r) X_r)|_{X_r}\\
&=& -(K_{X_r}+\D_r). 
\end{eqnarray*}
Hence we get $ -(K_{X_r}+\D_r) = \lam L + F$ by the assumption (\ref{Eqnarray: almost complete assumption}).

For any $v\in \Val^\circ_{X_r}$, let $a=A_{X_r,\D_r}(v), b=S^g(W_\bu; v)$ and $c=v(F)$. Then 
$$A_{X_r, \D_r +F}(v) = a-c,\quad  S(-(K_{X_r}+\D_r+F); v)= b-c. $$
By the elementary inequality (\ref{Eqnarray: elementary inequality}), we have either 
\begin{eqnarray*} 
\frac{a-c}{b-c} \ge \frac{a}{b} \ge 1
\quad \textrm{or} \quad
\frac{a-c}{b-c} <  \frac{a}{b} < 1. 
\end{eqnarray*}
Hence the equivalence of K-semistability holds. 

For the equivalence of K-polystability, we may assume that the two pairs are both K-semistable. In particular, $(X_r, \D_r'=\D_r+F)$ is a log Fano pair. If the maximal torus of $\Aut(X,\D)$ is of rank $r$, then we are done by Theorem \ref{Theorem: main thm}. Otherwise, applying Construction \ref{Construction: toric divisor sequence} to $(X_r,\D_r')$ we get a third pair $(X_{r+s}, \D_{r+s}')$ and a $\IN\times \IN^{r+s}$-graded linear series $W_\bu'$ on $X_{r+s}$. Arguing as Theorem \ref{Theorem: main thm} we see that the $g$-weighted K-polystability of $(X,\D)$ and the K-polystability of $(X_r,\D_r')$ are both equivalent to $\delta(X_{r+s},\D_{r+s}';W_\bu')>1$. 
\end{proof}

\subsection{Existence of $g$-solitons}

As an application of Theorem \ref{Theorem: main thm} and \ref{Theorem: main thm of almost complete}, we show the existence of $g$-solitons on some Fano $\IT$-varieties of complexity one, generalizing the corresponding results on toric Fano varieties \cite{WZ04,SZ12,BLXZ23}. 

Let $X$ be a Fano $\IT$-variety of complexity one (maximal torus of $\Aut(X)$ is of rank $r=\dim X-1$). Then $X_r\cong \IP^1$ is of Picard number one. In particular, $W^{X_r}_\bu$ satisfies (\ref{Eqnarray: almost complete assumption}) with respect to $\CO_{\IP^1}(1)$. 
Since $\D_r$ is obtained by taking different succeedingly starting at $\D_0=0$, by \cite[(4.4)]{Kol13}, it must be of the form $\D_r = \sum_i (1-\frac{1}{m_i})p_i$ for $m_i\in \IZ_{\ge 2}$ and $p_i\in \IP^1$. By Construction \ref{Construction: toric divisor sequence}, $(X_r,\D_r)$ is a Fano type klt pair. Hence $\Supp \D_r$ contains at most three points (otherwise $\sum_{1\le i\le 4} (1-\frac{1}{m_i}) \ge 2$). By Theorem \ref{Theorem: main thm of almost complete}, $-(K_{X_r}+\D_r+F)=\CO_{\IP^1}(\lam)$ is ample. 

\begin{thm}
Let $X$ be a Fano $\IT$-variety admitting a rank $r=\dim X-1$ torus action. If $\Supp \D_r$ consists of three points, then $X$ is $g$-weighted K-polystable for any weight function $g$. 
\end{thm}
\begin{proof}
We denote by $\D_r = \sum_{1\le i\le3} (1- \frac{1}{m_i})p_i$ and $F=\sum_{1\le i\le k} a_i p_i$, where $m_i\in \IZ_{\ge2}$ and $a_i\in \IR_{\ge0}$.
By Theorem \ref{Theorem: main thm of almost complete}, we have
$$0<\lam=\deg(-K_{\IP^1}-\D_r-F) =2 - \sum_{1\le i\le3} (1- \frac{1}{m_i}) -\sum_{1\le i\le k} a_i. $$ 
Hence $1+\sum_{i=1}^k a_i< \frac{1}{m_1}+\frac{1}{m_2}+\frac{1}{m_3}$. In particular, $\frac{1}{2}\le1-\frac{1}{m_i}+a_i < 1$ for any $1\le i\le 3$, and $a_i<\frac{1}{2}$ for any $i\ge 4$. Hence by \cite[Theorem 3]{Li15}, $(\IP^1, \D_r+F)$ is K-stable. The proof is finished by Theorem \ref{Theorem: main thm of almost complete}. 
\end{proof}

As the second application, we generalize some results of the existence of K\"ahler-Ricci solitons on Fano threefolds in \cite{MW23,MW24} to the existence of $g$-solitons for arbitrary weight function $g:\BP\to \IR_{>0}$. We need the following well-known result. 

\begin{lem}
\label{Lemma: curve GIT and K-stability}
Let $C\seq \IP^2$ or $\IP^1\times \IP^1$ be a plane cubic curve or a biconic curve. For any $0<c<1$, the pair $(\IP^2,cC)$ or $(\IP^1\times \IP^1,cC)$ is K-(semi/poly)stable if and only if $C$ is GIT-(semi/poly)stable. 
\end{lem}
\begin{proof}
It follows directly from \cite[Example 4.5]{ADL19} for plane cubic curves $C\seq\IP^2$, and one may prove for biconic curves $C\seq \IP^1\times\IP^1$ with the same argument. 
\end{proof}

\begin{ex}[Fano threefolds №2.28 and №3.14]\rm 
\label{Example: 2.28 and 3.14}
Let $H\seq \IP^3$ be a plane, and $C\seq H$ be a cubic curve (may not smooth). Let $X=\Bl_C \IP^3$ and $\tX =\Bl_p X$ for some $p\in \IP^3\setminus H$. Then $X$ is in №2.28 and $\tX$ is in №3.14 of Mori-Mukai's list when $C$ is smooth. The $\IG_m$-action on $\IP^3$ connecting $H$ and $p$ lifts to both $X,\tX$, and it is a maximal torus of $\Aut(X)$ and $\Aut(\tX)$. The strict transform $\tH\cong \IP^2$ of $H$ is a toric divisor of this $\IG_m$-action. 

By \cite[Section 6.2 and 6.4]{MW23}, the refinement of $-K_X$ and $-K_{\tX}$ by $\tH$ are 
\begin{eqnarray*}
W^{\tH}_{(1,\alpha)} = 
\left\{ \begin{array}{ll}
H^0\Big(\IP^2, \CO(3+2\alpha)\Big) 
& -1\le \alpha < 0,\\
H^0\Big(\IP^2, \CO(3-\alpha)\Big) + \alpha\cdot C
& 0\le \alpha\le \alpha_0, 
\end{array} \right. 
\end{eqnarray*} 
where $\alpha_0=3$ for $X$ and $\alpha_0=1$ for $\tX$. The moment polytope is $\BP=[-1,\alpha_0]$ and the DH measure is $\DH_\BP(d\alpha) = \vol(W^{\tH}_{(1,\alpha)})d\alpha$. The almost complete assumption (\ref{Eqnarray: almost complete assumption}) holds in this case. 

For $X$, we fix a weight function $g:[-1,3]\to \IR_{>0}$.  By (\ref{Eqnarray: Fut_g = 0}), we have
\begin{eqnarray}
\label{Eqnarray: Ex2.28}
\int_{-1}^{0}\alpha\cdot g(\alpha) (3+2\alpha)^2 d\alpha +
\int_{0}^{3}\alpha\cdot g(\alpha) (3-\alpha)^2 d\alpha = 0.  
\end{eqnarray} 
Then by Theorem \ref{Theorem: main thm of almost complete}, we see that $X$ is $g$-weighted K-polystable if and only if $(\IP^2,\mu C)$ is K-stable or K-polystable, where 
\begin{eqnarray*}
\mu &=& \frac{1}{\BV^g} \int_{0}^{3}\alpha\cdot g(\alpha) (3-\alpha)^2 d\alpha>0, \\
\BV^g &=& \int_{-1}^{0} g(\alpha) (3+2\alpha)^2 d\alpha +
\int_{0}^{3} g(\alpha) (3-\alpha)^2 d\alpha. 
\end{eqnarray*} 
Moreover, by (\ref{Eqnarray: Ex2.28}) we have 
$$\mu = \frac{1}{\BV^g}\int_{-1}^{0}(-\alpha)\cdot g(\alpha) (3+2\alpha)^2 d\alpha \le 
\frac{1}{\BV^g}\int_{-1}^{0} g(\alpha) (3+2\alpha)^2 d\alpha <1. $$
Hence by Lemma \ref{Lemma: curve GIT and K-stability}, we conclude that $X$ is $g$-weighted K-polystable if and only if $C$ is GIT-stable or polystable. This also holds for $\tX$. 
\end{ex}

\begin{rmk}\rm
These examples are generalization of \cite[Theorem 1.1 and 1.3]{MW23}, which say that for any continuous weight function $g:\BP\to \IR_{>0}$, the smooth Fano threefolds in №2.28 and №3.14 all admit $g$-soliton by \cite{HL23,BLXZ23} if the base field $\Ik=\IC$. In particular, if $g(\alpha) = e^{-\alpha\cdot\xi_0}$ for the soliton candidate $\xi_0\in N_\IR$, the $g$-solitons reveal the K\"ahler-Ricci solitons. 
\end{rmk}

Similar results hold for the optimal degenerations of Fano threefolds in №2.23(a). 

\begin{ex}[Optimal degenerations of Fano threefolds in №2.23(a)] \rm 
Let $Q_0\seq \IP^4$ be a cone over a smooth quadric surface $H\seq \IP^3$, and $C\seq H\cong \IP^1\times \IP^1$ be a biconic curve (i.e. $C\in |\CO_{\IP^1\times \IP^1}(2,2)|$). 
Then $X_0 = \Bl_C Q_0$ is the optimal degeneration of the K-unstable Fano variety $X=\Bl_C Q$ by \cite[Corollary 1.4]{MW24}, where $Q$ is a smoothing of $Q_0$ passing through $C$. The $\IG_m$-action of $Q_0$ along the cone direction lifts to $X_0$ naturally, and it is a maximal torus of $\Aut(X_0)$. 
The strict transform $\tH$ of $H$ is a toric divisor of this $\IT$-action.

By \cite[Section 3.3]{MW24}, the refinement of $-K_{X_0}$ by the toric divisor $\tH = \IP^1 \times \IP^1$ is 
\begin{eqnarray*}
W_{(1,\alpha)} = 
\left\{ \begin{array}{ll}
H^0\Big(\IP^1\times \IP^1, \CO(2+\alpha)\Big) 
& -1\le \alpha < 0,\\
H^0\Big(\IP^1\times \IP^1, \CO(2-\alpha)\Big) + \alpha\cdot C
& 0\le \alpha\le 2. 
\end{array} \right. 
\end{eqnarray*} 
where $C = H \cap E_C$. For any continuous weight function $g:\BP=[-1,2]\to \IR_{>0}$, with the same argument of Example \ref{Example: 2.28 and 3.14}, we see that $X_0$ is $g$-weighted K-polystable if and only if $C\seq \IP^1\times\IP^1$ is GIT-stable or polystable. 
\end{ex}

\begin{rmk}\rm
The key ingredients for the independence of the weight function $g$ in the above examples are that, firstly, the toric divisor $\tH\seq X$ has log discrepancy $1$; secondly, whenever $0<c<1$, the log Fano pairs $(\IP^2,cC_3)$ and $(\IP^1\times \IP^1, cC_{2,2})$ are K-stable for smooth $C$. 
\end{rmk}

However, the $g$-weighted K-stability depends on the choice of $g$ in general. 

\begin{ex}[Optimal degenerations of Fano threefolds in №2.23(b)] \rm 
Let $Q\seq \IP^4$ be a smooth quadric threefold and $H,H'\seq Q$ be mutually distinct singular hyperplane sections. Hence $H\cong H' \cong \IP(1,1,2)$ and they intersect at a smooth conic. Let $C_0 \seq Q$ be the non-reduced curve defined by $C_0=2H'|_H$, and $X_0=\Bl_{C_0} Q$. By \cite[Corollary 1.8]{MW24}, this is the optimal degeneration of the K-unstable Fano threefold $X=\Bl_C Q$ where $C=Q'|_H$ for some smooth quadric section $Q'$ not passing through the vertex of $H=\IP(1,1,2)$. There is a $\IG_m^2$-action on $Q$ under which $C_0$ is invariant, hence lifting to $X_0$. We consider the $\IG_m$-action on $X_0$ along the cone direction of $H$ and leaving each point of $H'$ invariant. Then the exceptional divisor $E$ obtained by blowing up $Q$ at the vertex of $H$ is a toric divisor of the $\IG_m$-action. 

By \cite[Section 3.5]{MW24}, the refinement of $-K_{X_0}$ by $E$ is  
\begin{eqnarray*}
W_{(1,\alpha)} = 
\left\{ \begin{array}{ll}
H^0\Big(\IP^2, \CO(3+\alpha)\Big) 
& -3\le \alpha < -2,\\
H^0\Big(\IP^2, \CO(\frac{1}{3}(5+\alpha))\Big) + \frac{1}{3}(2+\alpha)\cdot C_2
& -2\le \alpha\le 1, \\
H^0\Big(\IP^2, \CO(3-\alpha)\Big) + \alpha\cdot C_2
& 1\le \alpha\le 3. 
\end{array} \right. 
\end{eqnarray*} 
where $C_2=\tH|_{E}$ is a smooth plane conic. Let $g:\BP=[-3,3] \to \IR_{>0}$ be a continuous weight function. Then by (\ref{Eqnarray: Fut_g = 0}) we have 
\begin{eqnarray}
\label{Eqnarray: 2.23(b) Fut_g=0}
 \int_{-3}^{-2} \alpha \cdot g(\alpha) (3+\alpha)^2 d\alpha
+ \int_{-2}^{1} \alpha \cdot g(\alpha)\frac{1}{9}(5+\alpha)^2 d\alpha 
+ \int_1^3 \alpha \cdot g(\alpha) (3-\alpha)^2 d\alpha 
=0. 
\end{eqnarray} 
By Theorem \ref{Theorem: main thm of almost complete}, we see that $X_0$ is $g$-weighted K-polystable if and only if $(\IP^2, \mu C_2)$ is K-polystable, where 
\begin{eqnarray*}
\mu &=& \frac{1}{\BV^g} 
\Big(\int_{-2}^{1}\frac{1}{3}(2+\alpha)\cdot g(\alpha)\cdot \frac{1}{9}(5+\alpha)^2 d\alpha 
+ \int_1^3 \alpha \cdot g(\alpha) (3-\alpha)^2 d\alpha \Big) > 0, \\
\BV^g &=& \int_{-3}^{-2} g(\alpha) (3+\alpha)^2 d\alpha
+ \int_{-2}^{1} g(\alpha)\frac{1}{9}(5+\alpha)^2 d\alpha 
+ \int_1^3 g(\alpha) (3-\alpha)^2 d\alpha. 
\end{eqnarray*} 
Following from Example \ref{Example: plane conics}, we know that $(\IP^2, \mu C_2)$ is K-polystable if and only if $0<\mu<\frac{3}{4}$. Hence $X_0$ is $g$-weighted K-polystable for any continuous weight function $g:\BP\to \IR_{>0}$ such that $\mu<\frac{3}{4}$. 
In particular, if $g(\alpha) = e^{-\alpha\cdot\eta_0}$, where $\eta_0 \in \IR$ is determined by (\ref{Eqnarray: 2.23(b) Fut_g=0}). Then we can give an explicit estimate of $\mu$ and show that 
$\mu < 0.739237<\frac{3}{4}$. Hence $(X_0,\eta_0)$ is $g$-weighted K-polystable and admits a K\"ahler-Ricci soliton \cite[Theorem 1.7]{MW24}. 
\end{ex}


\section{Qdlt Fano type models}
\label{Section: Qdlt Fano type models}

In this section, we give another construction of $(X_r,\D_r,W^{X_r}_\bu)$ in Theorem \ref{Theorem: Intro. main thm} depending on the deep theory of higher rank finite generation developed by \cite{LXZ22,XZ22,Xu24}. 



Let $(X,\D)$ be a log Fano pair with a $\IT=\IG_m^r$-action, $M\cong \IZ^r$ and $N=M^\vee$ be the weight and co-weight lattices. 
For any face $F$ of the moment polytope $\BP\seq M_\IR$, we define its {\it normal cone} $\sigma_F\seq N_\IR$ by  
\begin{eqnarray*}
\sigma_F = \{\xi\in \IN_\IR\mid \la\alpha,\xi\ra=\la\alpha',\xi\ra >0,\, \forall \alpha,\alpha' \in F \}, 
\end{eqnarray*} 
which is a closed convex cone in $N_\IR$. 
The {\it moment fan} $\BF$ of the $\IT$-action is the set of normal cones corresponding to faces of $\BP$. Let $\sigma \in \BF$, then for all vectors $\xi$ in the relative interior of $\sigma$, the toric valuations $\wt_\xi$ have the same center $Z_\sigma$. Moreover, all the valuations in the boundary of $\sigma$ pass through $Z_\sigma$. 


\begin{thm} 
For any $\xi\in N_\IR$ of rational rank $r$, there exists an linearly independent sequence of primitive vectors $\xi_1,\cdots, \xi_r\in  N$ lying in a cone $\sigma \in \BF$, and a qdlt Fano type model $(Y,E=E_1+\cdots +E_r) \to (X,\D)$ such that $\xi=\sum_{1\le i\le r} a_i \xi_i$ for some $a_i>0$, and $\wt_{\xi_i}=\ord_{E_i}$. 
\end{thm}

\begin{proof}
This is an application of \cite[Theorem 3.14]{XZ22}. 
Since $\xi\in N_\IR$ is of rational rank $r$, there exists $\sigma \in \BF$ of dimension $r$ such that the interior of $\sigma$ containing $\xi$. Note that $\wt_\xi$ induces a product $r$-step degeneration of $(X,\D)$. By \cite[Lemma 4.3]{LXZ22}, there is a $\IT$-equivariant log smooth model $\tau: (W,F)\to (X,\D)$ such that $(X,\D)$ admits a special $\IQ$-complement $\Gamma_W$ with respect to $(W,F)$ and $\wt_\xi\in \QM(W,F)\cap \LC(X,\D+\Gamma_W)$. The intersection of $\QM(W,F)$ and $\sigma$ in $\Val_X$ leads to a sub-division of $\sigma$. Hence there exists an $r$-dimensional subcone $\xi \in \sigma_0 \seq \sigma$ such that $\wt$ induces an embedding $\sigma_0 \seq \QM(W, F)$. 
We may choose a linearly independent sequence of primitive vectors $\xi_1,\cdots, \xi_r\in  \sigma_0\cap N$
such that $\xi=\sum_{1\le i\le r}a_i\xi_i$ for some $\IQ$-linearly independence positive real numbers $a_1,\cdots,a_r \in \IR_{>0}$. Let $E_i$ be the toric divisor corresponding to $\xi_i$, then $\wt_{\xi_i} = \ord_{E_i}$. 

Hence by \cite[Lemma 3.17]{XZ22}, there exists a birational toroidal morphism $\rho: Z\to (W, F)$ and a $\IQ$-complement $\Gamma$ such that 
\begin{itemize}
\item $\rho$ extracts exactly the divisors $E_1,\cdots, E_r$; 
\item $\Gamma$ is special with respect to $(Z,E=E_1+\cdots+E_r)$; and 
\item $\QM(Z,E) = \LC(X,\D+\Gamma)$. 
\end{itemize}
Finally by \cite[Lemma 3.15]{XZ22}, we get a qdlt Fano type model $(Y,E=E_1+\cdots+E_r)$ of $(X,\D)$. 
\end{proof}

Now we could give a simpler construction of $(X_r,\D_r,W^{X_r}_\bu)$. 

\begin{const} \rm
Let $(Y,E=E_1+\cdots+E_r)\to (X,\D)$ be a qdlt Fano type model such that $\ord_{E_i} = \wt_{\xi_i}$ for linearly independent primitive vectors $\xi_1,\cdots,\xi_r \in \IN$. Then there exists an effective $\IQ$-divisor $D_0$ on $Y$ such that $(Y,D_0+E)$ is qdlt, $D_0+E\ge \pi_*^{-1}\D$, and $-(K_Y+D_0+E)$ is ample. We may take $E_i$-adjunction to $(Y,D_0+E)$ succeedingly and get qdlt Fano pairs $(X_i,D_i+E_{>i})$ where $X_i=E_1\cap\cdots\cap E_i$, $D_i=\Diff_{X_i}(D_{i-1})$ and $E_{>i} = (E_{i+1}+\cdots +E_r)|_{X_i}$. Hence $E_{>r}=0$ and $(X_r,D_r)$ is a log Fano pair. Under this sequence of adjunctions, $(Y,\D_0=\pi_*^{-1}\D - \sum_i \ord_{E_i}(\D)E_i)$ generates a sequence of klt Fano type pairs $(X_i,\D_i)$ where $\D_i=\Diff_{X_i}(\D_{i-1})\le D_i$. 

Since $X_{i+1}\seq (X_i, D_i+E_{>i})$ is a component of $E_{>i}$, it is of plt-type on $X_i$ by \cite[Lemma 2.3]{XZ22}. 
By taking $X_i$-refinement succeedingly, we get a $\IN\times \IN^i$-graded linear series $W^{X_i}_\bu$ on $X_i$. 
\end{const}


Following the argument of Lemma \ref{Lemma: plt Fano type model}, the triple $(X_r,\D_r,W^{X_r}_\bu)$ constructed above satisfies Lemma \ref{Lemma: A=A, S=S}. Hence Theorem \ref{Theorem: Intro. main thm} and \ref{Theorem: Intro. main thm. almost complete} holds for this $(X_r,\D_r,W^{X_r}_\bu)$.

\bibliographystyle{alpha}
\bibliography{ref}

\end{document}